%% file: 2015_ultra_fast_l1_deconv_Escande_Weiss.tex
\documentclass[11pt]{article}
\pdfoutput=1

\input{preamble}

\begin{document}
\input{Body/1-frontmatter.tex}

\newpage

\input{Body/5-intro.tex}
\input{Body/7-notation.tex}
\input{Body/10-theory.tex}
\input{Body/20-methods.tex}

\input{Body/30-numerical.tex}


\section*{Acknowledgments}

The authors wish to thank Manon Dugu\'e for a preliminary numerical study of the proposed idea. 
They thank Sandrine Anthoine, Caroline Chaux, Hans Feichtinger, Clothilde M\'elot and Bruno Torr\'esani for fruitful discussions and encouragements, which motivated the authors to work on this topic. 

\bibliographystyle{abbrv}
\bibliography{biblio}

\appendix
 
\end{document}

%% file: preamble.tex
\usepackage[utf8]{inputenc}

\usepackage[english]{babel}

\usepackage{geometry}

\usepackage{graphicx}
\graphicspath{{./images/}{./}}

\usepackage{caption}
\usepackage{subcaption}

\usepackage{url}
\usepackage{placeins}

\usepackage{longtable} 
\usepackage{booktabs} 

\usepackage{multirow}
\usepackage{rotating}

\usepackage{amsthm}
\usepackage{amsmath}\allowdisplaybreaks
\usepackage{amssymb}
\usepackage{mathrsfs}
\usepackage{dsfont}
\usepackage{amsbsy}
\usepackage{enumitem}
\usepackage{listings}

\usepackage[usenames,dvipsnames,svgnames,table]{xcolor}
\usepackage{xspace}
\usepackage{slashbox}

\usepackage{algorithm}
\usepackage{algpseudocode}

\usepackage{array}

\makeatletter
\newcommand{\thickhline}{%
    \noalign {\ifnum 0=`}\fi \hrule height 1pt
    \futurelet \reserved@a \@xhline
}
\newcolumntype{"}{@{\hskip\tabcolsep\vrule width 1pt\hskip\tabcolsep}}
\makeatother

\usepackage{pgfplots}
\pgfplotsset{compat=newest}
\usetikzlibrary{plotmarks}
\usepackage{grffile}

\usetikzlibrary{spy,calc}
\input{zoomboxarray.tex}

\newtheorem{theorem}{Theorem}
\newtheorem{lemma}[theorem]{Lemma}

\theoremstyle{definition}
\newtheorem{definition}{Definition}
\newtheorem{remark}{Remark}

\makeatletter
\let\save@mathaccent\mathaccent
\newcommand*\if@single[3]{%
  \setbox0\hbox{${\mathaccent"0362{#1}}^H$}%
  \setbox2\hbox{${\mathaccent"0362{\kern0pt#1}}^H$}%
  \ifdim\ht0=\ht2 #3\else #2\fi
  }
\newcommand*\rel@kern[1]{\kern#1\dimexpr\macc@kerna}
\newcommand*\widebar[1]{\@ifnextchar^{{\wide@bar{#1}{0}}}{\wide@bar{#1}{1}}}
\newcommand*\wide@bar[2]{\if@single{#1}{\wide@bar@{#1}{#2}{1}}{\wide@bar@{#1}{#2}{2}}}
\newcommand*\wide@bar@[3]{%
  \begingroup
  \def\mathaccent##1##2{%
    \let\mathaccent\save@mathaccent
    \if#32 \let\macc@nucleus\first@char \fi
    \setbox\z@\hbox{$\macc@style{\macc@nucleus}_{}$}%
    \setbox\tw@\hbox{$\macc@style{\macc@nucleus}{}_{}$}%
    \dimen@\wd\tw@
    \advance\dimen@-\wd\z@
    \divide\dimen@ 3
    \@tempdima\wd\tw@
    \advance\@tempdima-\scriptspace
    \divide\@tempdima 10
    \advance\dimen@-\@tempdima
    \ifdim\dimen@>\z@ \dimen@0pt\fi
    \rel@kern{0.6}\kern-\dimen@
    \if#31
      \overline{\rel@kern{-0.6}\kern\dimen@\macc@nucleus\rel@kern{0.4}\kern\dimen@}%
      \advance\dimen@0.4\dimexpr\macc@kerna
      \let\final@kern#2%
      \ifdim\dimen@<\z@ \let\final@kern1\fi
      \if\final@kern1 \kern-\dimen@\fi
    \else
      \overline{\rel@kern{-0.6}\kern\dimen@#1}%
    \fi
  }%
  \macc@depth\@ne
  \let\math@bgroup\@empty \let\math@egroup\macc@set@skewchar
  \mathsurround\z@ \frozen@everymath{\mathgroup\macc@group\relax}%
  \macc@set@skewchar\relax
  \let\mathaccentV\macc@nested@a
  \if#31
    \macc@nested@a\relax111{#1}%
  \else
    \def\gobble@till@marker##1\endmarker{}%
    \futurelet\first@char\gobble@till@marker#1\endmarker
    \ifcat\noexpand\first@char A\else
      \def\first@char{}%
    \fi
    \macc@nested@a\relax111{\first@char}%
  \fi
  \endgroup
}
\makeatother

\newcommand{\R}{ \mathbb{R} }
\newcommand{\T}{ \mathbb{T} }

\newcommand{\Z}{ \mathbb{Z} }
\newcommand{\N}{ \mathbb{N} }

\newcommand{\LL}{ L }

\def \bPsi{\mathbf{\Psi}}

\def \balpha{\alpha}
\def \be{e}
\def \bzero{0}
\def \bm{m}

\newcommand{\norm}[1]{\left\Vert #1\right\Vert}
\newcommand{\abs}[1]{\left\vert #1\right\vert}
\newcommand{\opnorm}[1]{{\left\vert\kern-0.25ex\left\vert\kern-0.25ex\left\vert #1 \right\vert\kern-0.25ex\right\vert\kern-0.25ex\right\vert}}
\newcommand{\dotproduct}[2]{\ensuremath{\left\langle #1, #2 \right\rangle}\xspace}
\newcommand{\p}{\partial}

\DeclareMathOperator{\prox}{Prox}
\newcommand{\argmin}{\mathop{\mathrm{arg\,min}}}

\newcommand{\supp}{\mathop{\mathrm{supp}}}

\newcommand{\set}[1]{\left\{ #1 \right\}}

\newcommand{\diag}{\mathrm{diag}}

\newcommand{\fa}{ \textrm{for all }}

\newcommand{\blue}[1]{#1}
\usepackage{hyperref}
\usepackage[squaren,Gray]{SIunits}

%% file: zoomboxarray.tex
\newif\ifblackandwhitecycle
\gdef\patternnumber{0}

\pgfkeys{/tikz/.cd,
    zoombox paths/.style={
        draw=orange,
        very thick
    },
    black and white/.is choice,
    black and white/.default=static,
    black and white/static/.style={ 
        draw=white,   
        zoombox paths/.append style={
            draw=white,
            postaction={
                draw=black,
                loosely dashed
            }
        }
    },
    black and white/static/.code={
        \gdef\patternnumber{1}
    },
    black and white/cycle/.code={
        \blackandwhitecycletrue
        \gdef\patternnumber{1}
    },
    black and white pattern/.is choice,
    black and white pattern/0/.style={},
    black and white pattern/1/.style={    
            draw=white,
            postaction={
                draw=black,
                dash pattern=on 2pt off 2pt
            }
    },
    black and white pattern/2/.style={    
            draw=white,
            postaction={
                draw=black,
                dash pattern=on 4pt off 4pt
            }
    },
    black and white pattern/3/.style={    
            draw=white,
            postaction={
                draw=black,
                dash pattern=on 4pt off 4pt on 1pt off 4pt
            }
    },
    black and white pattern/4/.style={    
            draw=white,
            postaction={
                draw=black,
                dash pattern=on 4pt off 2pt on 2 pt off 2pt on 2 pt off 2pt
            }
    },
    zoomboxarray inner gap/.initial=5pt,
    zoomboxarray columns/.initial=2,
    zoomboxarray rows/.initial=2,
    subfigurename/.initial={},
    figurename/.initial={zoombox},
    zoomboxarray/.style={
        execute at begin picture={
            \begin{scope}[
                spy using outlines={%
                    zoombox paths,
                    width=\imagewidth / \pgfkeysvalueof{/tikz/zoomboxarray columns} - (\pgfkeysvalueof{/tikz/zoomboxarray columns} - 1) / \pgfkeysvalueof{/tikz/zoomboxarray columns} * \pgfkeysvalueof{/tikz/zoomboxarray inner gap} -\pgflinewidth,
                    height=\imageheight / \pgfkeysvalueof{/tikz/zoomboxarray rows} - (\pgfkeysvalueof{/tikz/zoomboxarray rows} - 1) / \pgfkeysvalueof{/tikz/zoomboxarray rows} * \pgfkeysvalueof{/tikz/zoomboxarray inner gap}-\pgflinewidth,
                    magnification=3,
                    every spy on node/.style={
                        zoombox paths
                    },
                    every spy in node/.style={
                        zoombox paths
                    }
                }
            ]
        },
        execute at end picture={
            \end{scope}
            \node at (image.north) [anchor=north,inner sep=0pt] {\subcaptionbox{\label{\pgfkeysvalueof{/tikz/figurename}-image}}{\phantomimage}};
            \node at (zoomboxes container.north) [anchor=north,inner sep=0pt] {\subcaptionbox{\label{\pgfkeysvalueof{/tikz/figurename}-zoom}}{\phantomimage}};
     \gdef\patternnumber{0}
        },
        spymargin/.initial=0.5em,
        zoomboxes xshift/.initial=1,
        zoomboxes right/.code=\pgfkeys{/tikz/zoomboxes xshift=1},
        zoomboxes left/.code=\pgfkeys{/tikz/zoomboxes xshift=-1},
        zoomboxes yshift/.initial=0,
        zoomboxes above/.code={
            \pgfkeys{/tikz/zoomboxes yshift=1},
            \pgfkeys{/tikz/zoomboxes xshift=0}
        },
        zoomboxes below/.code={
            \pgfkeys{/tikz/zoomboxes yshift=-1},
            \pgfkeys{/tikz/zoomboxes xshift=0}
        },
        caption margin/.initial=4ex,
    },
    adjust caption spacing/.code={},
    image container/.style={
        inner sep=0pt,
        at=(image.north),
        anchor=north,
        adjust caption spacing
    },
    zoomboxes container/.style={
        inner sep=0pt,
        at=(image.north),
        anchor=north,
        name=zoomboxes container,
        xshift=\pgfkeysvalueof{/tikz/zoomboxes xshift}*(\imagewidth+\pgfkeysvalueof{/tikz/spymargin}),
        yshift=\pgfkeysvalueof{/tikz/zoomboxes yshift}*(\imageheight+\pgfkeysvalueof{/tikz/spymargin}+\pgfkeysvalueof{/tikz/caption margin}),
        adjust caption spacing
    },
    calculate dimensions/.code={
        \pgfpointdiff{\pgfpointanchor{image}{south west} }{\pgfpointanchor{image}{north east} }
        \pgfgetlastxy{\imagewidth}{\imageheight}
        \global\let\imagewidth=\imagewidth
        \global\let\imageheight=\imageheight
        \gdef\columncount{1}
        \gdef\rowcount{1}
        
    },
    image node/.style={
        inner sep=0pt,
        name=image,
        anchor=south west,
        append after command={
            [calculate dimensions]
            node [image container,subfigurename=\pgfkeysvalueof{/tikz/figurename}-image] {\phantomimage}
            node [zoomboxes container,subfigurename=\pgfkeysvalueof{/tikz/figurename}-zoom] {\phantomimage}
        }
    },
    color code/.style={
        zoombox paths/.append style={draw=#1}
    },
    connect zoomboxes/.style={
    spy connection path={\draw[draw=none,zoombox paths] (tikzspyonnode) -- (tikzspyinnode);}
    },
    help grid code/.code={
        \begin{scope}[
                x={(image.south east)},
                y={(image.north west)},
                font=\footnotesize,
                help lines,
                overlay
            ]
            \foreach \x in {0,1,...,9} { 
                \draw(\x/10,0) -- (\x/10,1);
                \node [anchor=north] at (\x/10,0) {0.\x};
            }
            \foreach \y in {0,1,...,9} {
                \draw(0,\y/10) -- (1,\y/10);                        \node [anchor=east] at (0,\y/10) {0.\y};
            }
        \end{scope}    
    },
    help grid/.style={
        append after command={
            [help grid code]
        }
    },
}

\newcommand\phantomimage{%
    \phantom{%
        \rule{\imagewidth}{\imageheight}%
    }%
}
\newcommand\zoombox[2][]{
    \begin{scope}[zoombox paths]
        \pgfmathsetmacro\xpos{
            (\columncount-1)*(\imagewidth / \pgfkeysvalueof{/tikz/zoomboxarray columns} + \pgfkeysvalueof{/tikz/zoomboxarray inner gap} / \pgfkeysvalueof{/tikz/zoomboxarray columns} ) + \pgflinewidth
        }
        \pgfmathsetmacro\ypos{
            (\rowcount-1)*( \imageheight / \pgfkeysvalueof{/tikz/zoomboxarray rows} + \pgfkeysvalueof{/tikz/zoomboxarray inner gap} / \pgfkeysvalueof{/tikz/zoomboxarray rows} ) + 0.5*\pgflinewidth
        }
        \edef\dospy{\noexpand\spy [
            #1,
            zoombox paths/.append style={
                black and white pattern=\patternnumber
            },
            every spy on node/.append style={#1},
            x=\imagewidth,
            y=\imageheight
        ] on (#2) in node [anchor=north west] at ($(zoomboxes container.north west)+(\xpos pt,-\ypos pt)$);}
        \dospy
        \pgfmathtruncatemacro\pgfmathresult{ifthenelse(\columncount==\pgfkeysvalueof{/tikz/zoomboxarray columns},\rowcount+1,\rowcount)}
        \global\let\rowcount=\pgfmathresult
        \pgfmathtruncatemacro\pgfmathresult{ifthenelse(\columncount==\pgfkeysvalueof{/tikz/zoomboxarray columns},1,\columncount+1)}
        \global\let\columncount=\pgfmathresult
        \ifblackandwhitecycle
            \pgfmathtruncatemacro{\newpatternnumber}{\patternnumber+1}
            \global\edef\patternnumber{\newpatternnumber}
        \fi
    \end{scope}
}

%% file: Body/1-frontmatter.tex
\title{Real-time $\ell^1-\ell^2$ deblurring using wavelet expansions of operators}

\author{Paul Escande \footnotemark[1] \and Pierre Weiss \footnotemark[2] }

\maketitle

\renewcommand{\thefootnote}{\fnsymbol{footnote}}
\footnotetext[1]{D\'epartement d'Ing\'enierie des Syst\`emes Complexes (DISC), Institut Sup\'erieur de l'A\'eronautique et de l'Espace (ISAE), Toulouse, France,  \url{paul.escande@gmail.com} }
\footnotetext[2]{Institut des Technologies Avanc\'ees en Sciences du Vivant, ITAV-USR3505 and Institut de Math\'ematiques de Toulouse, IMT-UMR5219, CNRS and universit\'e de Toulouse, Toulouse, France, \url{pierre.armand.weiss@gmail.com}}

\begin{abstract}
Image deblurring is a fundamental problem in imaging, usually solved with computationally intensive optimization procedures. 
We show that the minimization can be significantly accelerated by leveraging the fact that images and blur operators are compressible in the same orthogonal wavelet basis. 
The proposed methodology consists of three ingredients: i) a sparse approximation of the blur operator in wavelet bases, ii) a diagonal preconditioner and iii) an implementation on massively parallel architectures. 
Combing the three ingredients leads to acceleration factors ranging from $30$ to $250$ on a typical workstation. 
For instance, a $1024\times 1024$ image can be deblurred in $0.15$ seconds, which corresponds to real-time.
\end{abstract}

\textbf{Keywords:} sparse wavelet expansion, preconditioning, GPU programming, image deblurring, inverse problems.

\medskip

\textbf{AMS:} 65F50, 65R30, 65T60, 65Y20, 42C40, 45Q05, 45P05, 47A58

%% file: Body/5-intro.tex
\section{Introduction} \label{sec:introduction}

Most imaging devices produce blurry images. 
This degradation very often prevents to correctly interpret image contents and sometimes ruins expensive experiments.
One of the most advertised examples of that type is Hubble space telescope \footnote{The total cost of Hubble telescope is estimated at 10 billions US Dollars \cite{HubblePrice}.}, which was discovered to suffer from severe optical aberrations after being launched. Such situations occur on a daily basis in fields such as biomedical imaging, astronomy or conventional photography. 

Starting from the seventies, a large number of numerical methods to deblur images was therefore developed. 
The first methods were based on linear estimators such as the Wiener filter \cite{wiener1949extrapolation}. 
They were progressively replaced by more complicated nonlinear methods, incorporating prior knowledge on the image contents. 
We refer the interested reader to the following review papers  \cite{starck2002deconvolution,wang2014recent,pustelnik2015wavelet} to get an overview of the available techniques.

Despite providing better reconstruction results, the most efficient methods are often disregarded in practice, due to their high computational complexity, especially for large 2D or 3D images. 
The goal of this paper is to develop new numerical strategies that significantly reduce the computational burden of image deblurring. The proposed ideas yield a \emph{fast deblurring method}, compatible with \emph{large data} and \emph{routine use}. They allow handling both \emph{stationary and spatially varying blurs}. The proposed algorithm does not reach the state-of-the-art in terms of image quality, because the prior is too simple, but still performs well in short computing times.

\subsection{Image formation and image restoration models}

In this paper, we assume that the observed image $u_0$ reads:
\begin{equation}
 u_0=Hu +b,
\end{equation}
where $u\in \R^N$ is the clean image we wish to recover, $b\sim \mathcal{N}(0,\sigma^2 I_N)$ is a white Gaussian noise of standard deviation $\sigma$ and $H\in \R^{N\times N}$ is a blurring operator. Loosely speaking, a blurring operator replaces the value of a pixel by a mean of its neighbours. A precise definition will be given in Section \ref{sec:theory}. 

Let $\Psi\in \R^{N\times N}$ denote an orthogonal wavelet transform and let $A=H\Psi$.
A standard variational formulation to restore $u$ consists of solving:
\begin{equation}\label{eq:l1problem}
 \min_{x\in \R^N} E(x) = F(x) + G(x).
\end{equation}
In this equation, $F(x)$ is a quadratic data fidelity term defined by 
\begin{equation}\label{eq:defF}
F(x)=\frac{1}{2}\|A x- u_0\|_2^2.
\end{equation}
The regularization term $G(x)$ is defined by:
\begin{equation}\label{def:G}
G(x) = \|x\|_{1,w}=\sum_{i=1}^N w[i]|x[i]| 
\end{equation}
The vector of weights $w\in \R^N$ is a regularization parameter that may vary across subbands of the wavelet transform. 
The weighted $\ell^1$-norm is well known to promote sparse vectors. 
This is usually advantageous since images are compressible in the wavelet domain.
Overall, problem \eqref{eq:l1problem} consists of finding an image $\Psi x$ consistent with the observed data $u_0$ with a sparse representation $x$ in the wavelet domain.


Many groups worldwide have proposed minimizing similar cost functions in the literature, see e.g.  \cite{figueiredo2003algorithm,neelamani2004forward,vonesch2008fast,vonesch2009fast}. 
The current trend is to use redundant dictionaries $\Psi$ such as the undecimated wavelet transforms or learned transforms instead of orthogonal transforms \cite{starck2003wavelets,chaux2007variational,chai2007deconvolution,cai2010framelet}. 
This usually allows reducing reconstruction artifacts.
We focus here on the case where $\Psi$ is orthogonal. 
This property will help designing much faster algorithms.

\subsection{Standard optimization algorithms}

A lot of algorithms based on proximity operators were designed in the last decade to solve convex problems of type \eqref{eq:l1problem}.
We refer the reader to the review papers \cite{beck2009gradient,combettes2011proximal} to get an overview of the available techniques.
A typical method is the accelerated proximal gradient descent, also known as FISTA (Fast Iterative Soft Thresholding Algorithm) \cite{beck2009fast}. By letting $\|A\|_2$ denote the largest singular value of $A$, it takes the form described in Algorithm \ref{alg:FISTA}. This method got very popular lately due to its ease of implementation and relatively fast convergence.

\begin{algorithm}
\caption{Accelerated proximal gradient descent}
\label{alg:FISTA}
\begin{algorithmic}[1]
\State \textbf{input:} Initial guess $x^{(0)}=y^{(1)}$, $\tau=1/\|A\|_2^2$ and $Nit$.
\For{$k=1$ to $Nit$}
\State Compute $\nabla F(y^{(k)}) = A^*(A y^{(k)} -u_0 )$. \Comment{\textcolor{red}{\textbf{99.35$''$}}}
\State $x^{(k)} = \prox_{\tau G } \left( y^{(k)} - \tau \nabla F(y^{(k)}) \right)$. \Comment{2.7$''$}
\State  $y^{(k+1)} = x^{(k)} + \frac{k-1}{k+2} (x^{(k)} - x^{(k-1)})$. \Comment{1.1$''$}
\EndFor
\end{algorithmic}
\end{algorithm}

Let us illustrate this method on a practical deconvolution experiment. We use a $1024\times 1024$ image and assume that $H u = h\star u$, where $\star$ denotes the discrete convolution product and $h$ is a motion blur described on Figure \ref{fig:PSF_movement}. In practice, the PSNR of the deblurred image stabilizes after 500 iterations. The computing times on a workstation with Matlab and mex-files is around 103$''$15. The result is shown on Figure \ref{fig:deconv_book}. Profiling the code leads to the computing times shown on the right-hand-side of Algorithm \ref{alg:FISTA}. As can be seen, 96\% of the computing time is spent in the gradient evaluation. This requires computing two wavelet transforms and two fast Fourier transforms. 
This simple experiment reveals that two approaches can be used to reduce computing times:
\begin{itemize}
 \item \emph{Accelerate gradients computation.}
 \item \emph{Use more sophisticated minimization algorithms to accelerate convergence.} 
\end{itemize}

\subsection{The proposed ideas in a nutshell}

The method proposed in this paper relies on three ideas. First, function $F$ in equation \eqref{eq:defF} can be approximated by another function $F_K$ such that $\nabla F_K$ is inexpensive to compute. 
Second, we characterize precisely the structure of the Hessian of $F_K$, allowing to design efficient preconditioners.
Finally, we implement the iterative algorithm on a GPU.
The first two ideas, which constitute the main contribution of this paper, are motivated by our recent observation that spatially varying blur operators are compressible and have a well characterized structure in the wavelet domain \cite{escande2015sparse}. We showed that matrix 
\begin{equation}
 \Theta=\Psi^* H \Psi,
\end{equation}
which differs from $H$ by a change of basis, has a particular banded structure, with many negligible entries. 
Therefore, one can construct a $K$-sparse matrix $\Theta_K$ (i.e. a matrix with at most $K$ non zero entries) such that $\Theta_K\simeq \Theta$.

\paragraph{Problem approximation.}

Using the fact that $\Psi$ is an orthogonal transform allows writing that:
\begin{align}
  &\min_{x\in \R^N} \frac{1}{2}\|H\Psi x- u_0\|_2^2 + \|x\|_{1,w} \\
  & = \min_{x\in \R^N} \frac{1}{2}\|\Psi^* ( H\Psi x- u_0) \|_2^2 + \|x\|_{1,w} \\
 &=\min_{x\in \R^N} \frac{1}{2}\|\Theta x- x_0 \|_2^2 + \|x\|_{1,w}, \label{eq:l1problemTheta}
\end{align}
where $x_0=\Psi^*u_0$ is the wavelet decomposition of $u_0$. Problem \eqref{eq:l1problemTheta} is expressed entirely in the wavelet domain, contrarily to problem \eqref{eq:l1problem}. However, matrix-vector products with $\Theta$ might be computationally expensive. We therefore approximate the variational problem \eqref{eq:l1problemTheta} by:
\begin{equation}
\min_{x\in \R^N} \frac{1}{2}\|\Theta_K x- x_0 \|_2^2 + \|x\|_{1,w}. \label{eq:l1problemThetaApprox} 
\end{equation}

Now, let $F_K(x)=\frac{1}{2}\|\Theta_K x- x_0 \|_2^2$. The gradient of $F_K$ reads:
\begin{equation}\label{eq:gradientapprox}
\nabla F_K(x) = \Theta_K^* (\Theta_K - x_0),
\end{equation}
and therefore requires computing two matrix-vector products with sparse matrices. 
Computing the approximate gradient \eqref{eq:gradientapprox} is usually much cheaper than computing $\nabla F(x)$ exactly using Fast Fourier transforms and fast wavelet transforms. This may come as a surprise since they are respectively of complexity $O(N\log(N))$ and $O(N)$. In fact, we will see that in favorable cases, the evaluation of $\nabla F_K(x)$ may require about $2$ operations per pixel!

\paragraph{Preconditioning.}
The second ingredient of our method relies on the observation that the Hessian matrix $H_{F_K} (x) = \Theta_K^* \Theta_K$ has a near diagonal structure with decreasing entries on the diagonal. 
This allows designing \emph{efficient preconditioners}, which reduces the number of iterations necessary to reach a satisfactory precision. In practice preconditioning leads to acceleration factors ranging from $2$ to $5$.

\paragraph{GPU implementation}

Finally, using massively parallel programming on graphic cards still leads to an acceleration factor of order $10$ on an NVIDIA K20c.
Of course, this factor could be improved further by using more performant graphic cards.
Combining the three ingredients leads to algorithms that are from $30$ to $250$ times faster than FISTA algorithm applied to \eqref{eq:l1problem}, which arguably constitutes the current state-of-the-art.

\subsection{Related works}

The idea of characterizing integral operators in the wavelet domain appeared nearly at the same time as wavelets, at the end of the eighties. 
Y. Meyer characterized many properties of Calder\'on-Zygmund operators in his seminal book \cite{meyer1989wavelets}. 
Later, Beylkin, Coifman and Rokhlin \cite{beylkin1991fast}, showed that those theoretical results may have important consequences for the fast resolution of partial differential equations and the compression of matrices.
Since then, the idea of using multiscale representations has been used extensively in numerical simulation of physical phenomena. The interested reader can refer to \cite{cohen2003numerical} for some applications. 

Quite surprisingly, it seems that very few researchers attempted to apply them in imaging. 
In \cite{chang2000wavelet,Malgouyres2002309}, the authors proposed to approximate integral operators by matrices diagonal in the wavelet domain. Our experience is that diagonal approximations are too crude to provide sufficiently good approximations. More recently the authors of \cite{wei2014fast,escande2015sparse} proposed independently to compress operators in the wavelet domain. However they did not explore its implications for the fast resolution of inverse problems.

On the side of preconditioning, the two references \cite{vonesch2008fast,vonesch2009fast} are closely related to our work. The authors designed a few preconditioners to accelerate the convergence of the proximal gradient descent (also called thresholded Landweber algorithm or Iterative Soft Thresholding Algorithm). 
Overall, the idea of preconditioning is therefore not new.
To the best of our knowledge, our contribution is however the first that is based on a precise understanding of the structure of $\Theta$.

\subsection{Paper outline}

The paper is structured as follows. We first provide some notation and definitions in section \ref{sec:notation}. 
We then provide a few results characterizing the structure of blurring operators in section \ref{sec:theory}. 
We propose two simple explicit preconditioners in section \ref{sec:optimization}.
Finally, we perform numerical experiments and comparisons in section \ref{sec:numerical}.

%% file: Body/7-notation.tex
\section{Notation} \label{sec:notation}

In this paper, we consider $d$ dimensional images.
To simplify the discussion, we use circular boundary conditions and work on the $d$-dimensional torus $\Omega=\T^d$, where $\T^d=\R^d \backslash \Z^d$. 
The space $\LL^2(\Omega)$ denotes the space of squared integrable functions defined on $\Omega$. 

Let $\balpha=(\alpha_1,\hdots, \alpha_d)$ denote a multi-index. 
The sum of its components is denoted $|\balpha|=\sum_{i=1}^d \alpha_i$. 
The Sobolev spaces $W^{M,p}$ are defined as the set of functions $f \in \LL^p$ with partial derivatives up to order $M$ in $\LL^p$ where $p \in [1, + \infty]$ and $M \in \N$. 
These spaces, equipped with the following norm are Banach spaces
\begin{equation}
  \norm{f}_{W^{M,p}} = \norm{f}_{\LL^p} + \abs{f}_{W^{M,p}}, \quad \text{ where,} \quad \abs{f}_{W^{M,p}} = \sum_{\abs{\balpha} = M } \norm{\p^{\balpha} f}_{\LL^p},
\end{equation}
where $\p^{\balpha} f = \frac{\partial^{\alpha_1}}{\partial x_1^{\alpha_1}}\hdots \frac{\partial^{\alpha_d}}{\partial x_d^{\alpha_d}} f$.


Let us now define a wavelet basis on $\LL^2(\Omega)$. 
To this end, we first introduce a 1D wavelet basis on $\T$. 
Let $\phi$ and $\psi$ denote the scaling and mother wavelets and assume that the mother wavelet $\psi$ has $M$ vanishing moments, i.e. 
\begin{equation}
	\fa 0 \leq m < M, \quad \int_{[0,1]} t^m \psi(t) dt = 0.
\end{equation}

Translated and dilated versions of the wavelets are defined, for $j \geq 0$, as follows
\begin{equation}
  \phi_{j,l} = 2^{j/2} \phi\left( 2^{j} \cdot - l \right),
\end{equation}
\begin{equation}\label{eq:defwavelets}
  \psi_{j,l} = 2^{j/2} \psi\left( 2^{j} \cdot - l \right),
\end{equation}
with $l \in \mathcal{T}_j$ and $\mathcal{T}_j = \{0,\ldots,2^j-1\}$.

In dimension $d$, we use isotropic separable wavelet bases, see, e.g., \cite[Theorem 7.26, p. 348]{Mallat-Book}. 
Let $\bm=(m_1,\hdots,m_d)$.
Define $\rho_{j,l}^0=\phi_{j,l}$ and $\rho_{j,l}^1=\psi_{j,l}$. 
Let $\be=(e_1,\hdots,e_d)\in \{0,1\}^d$. 
For the ease of reading, we will use the shorthand notation $\lambda = (j,m,e)$ and $|\lambda|=j$. 
We also let 
\begin{equation}
\Lambda_0 = \set{ (j,m,e) \; | \; j \in \Z, \; m \in \mathcal{T}_j, \; e \in \set{0,1}^d} 
\end{equation}
 and  
\begin{equation}
\Lambda = \set{ (j,m,e) \; | \; j \in \Z, \; m \in \mathcal{T}_j, \; e \in \set{0,1}^d \setminus \{0\} }. 
\end{equation}
Wavelet $\blue{ \psi_\lambda} $ is defined by $ \blue{\psi_{\lambda}(x_1, \ldots, x_d)} = \psi_{j,\bm}^\be(x_1,\hdots,x_d)=\rho_{j,m_1}^{e_1}(x_1)\hdots \rho_{j,m_d}^{e_d}(x_d)$.
Elements of the separable wavelet basis consist of tensor products of scaling and mother wavelets at the same scale. 
Note that if $\be\neq \bzero$ wavelet $\psi_{j,\bm}^\be$ has $M$ vanishing moments in $\R^d$.
We let $\displaystyle I_{j,m}=\cup_{e} \supp{\psi^e_{j,m}}$ and $I_{\lambda} = \supp{ \psi_{\lambda}}$.
The distance between the supports of $\psi_{\lambda}$ and $\psi_{\mu}$ is defined by
\begin{align}
	\mathrm{dist}(I_{\lambda},I_{\mu}) &= \inf_{x\in I_{\lambda},\, y\in I_{\mu}} \|x-y\|_\infty \\
	& = \max \left(0, \norm{2^{-j}m - 2^{-k}n}_\infty - (2^{-j} + 2^{-k})\frac{c(M)}{2}\right).
\end{align}

With these definitions, every function $f\in\LL^2(\Omega)$ can be written as
\begin{align}
 u & = \dotproduct{u}{\psi^0_{0,0}} \psi^0_{0,0} + \sum_{e\in \{0,1\}^d \setminus \{0\}} \sum_{j=0}^{+\infty}\sum_{m \in \mathcal{T}_j} \dotproduct{u}{\psi^e_{j,m}} \psi^e_{j,m} \\
	& \blue{ = \dotproduct{u}{\psi^0_{0,0}} \psi^0_{0,0} + \sum_{ \lambda \in \Lambda} \dotproduct{u}{\psi_{\lambda}} \psi_\lambda } \\
	& \blue{ = \sum_{ \lambda \in \Lambda_0} \dotproduct{u}{\psi_{\lambda}} \psi_\lambda }.
\end{align}

Finally, we let $\Psi^* : \LL^2(\Omega)  \rightarrow  l^2(\Z)$ denote the wavelet decomposition operator and $\Psi : l^2(\Z) \rightarrow \LL^2(\Omega)$ its associated reconstruction operator. 
The discrete wavelet transform is denoted $\bPsi:\R^N\to \R^N$. 
We refer to \cite{Mallat-Book, daubechies_ten_1992,cohen1993wavelets} for more details on the construction of wavelet bases.

%% file: Body/10-theory.tex
\section{Wavelet decompositions of blurring operators}\label{sec:theory}

In this section we remind some results on the decomposition of blurring operators in the wavelet domain.

\subsection{Definition of blurring operators}

A blurring operator $H$ can be modelled by a linear integral operator $H : \LL^2(\Omega) \to \LL^2(\Omega)$:
\begin{equation} \label{eq:int_op}
	\forall x \in \Omega, \quad Hu(x) = \int_{\Omega} K(x,y) u(y) dy.
\end{equation}
The function $K : \Omega \times \Omega \to \R$ is called kernel of the integral operator and defines the Point Spread Function (PSF) $K(\cdot, y)$ at location $y \in \Omega$. The image $Hu$ is the blurred version of $u$. 
Following our recent paper \cite{escande2015sparse}, we define blurring operators as follows.

\begin{definition}[Blurring operators \cite{escande2015sparse}] \label{def:blurring_operators}
	Let $M\in \N$ and $f:[0,1]\to \R_+$ denote a non-increasing bounded function.
	An integral operator is called a blurring operator in the class $\mathcal{A}(M,f)$ if it satisfies the following properties: 
	\begin{enumerate}
		\item Its kernel $K\in W^{M,\infty}(\Omega\times \Omega)$; 
		\item The partial derivatives of $K$ satisfy:
		\begin{enumerate}
		  \item \label{def:blurring_operators:PSFsmoothness} 
		  \begin{equation} \label{eq:blurring_operators:dx_decay}
			  \forall \abs{\alpha} \leq M, \ \forall (x,y) \in \Omega\times \Omega, \quad \abs{\p_x^\alpha K(x,y)} \leq f\left( \norm{x-y}_\infty \right),
		  \end{equation}
		  \item \label{def:blurring_operators:PSFVariationSmoothness}
		  \begin{equation} \label{eq:blurring_operators:dy_decay}
			  \forall \abs{\alpha} \leq M, \ \forall (x,y) \in \Omega\times \Omega, \quad \abs{\p_y^\alpha K(x,y)} \leq f \left( \norm{x-y}_\infty \right).
		  \end{equation}
	  \end{enumerate}
	\end{enumerate}
\end{definition}

Condition \eqref{eq:blurring_operators:dx_decay} means that the PSF is smooth, while condition \eqref{eq:blurring_operators:dy_decay} indicates that the PSFs vary smoothly. 
These regularity assumptions are met in a large number of practical problems. 
In addition, they allow deriving theorems similar to those of the seminal papers of Y. Meyer, R. Coifman, G. Beylkin and V. Rokhlin \cite{meyer2000wavelets, beylkin1991fast}. Those results basically state that an operator in the class $\mathcal{A}(M,f)$ can be represented and computed efficiently when decomposed in a wavelet basis. We make this key idea precise in the next paragraph.

\subsection{Decomposition in wavelet bases}

Since $H$ is defined on the Hilbert space $\LL^2(\Omega)$, it can be written as $H = \Psi \Theta \Psi^*$, where $\Theta : \ell^2(\Z) \to \ell^2(\Z)$ is the infinite matrix representation of the operator $H$ in the wavelet domain. 
Matrix $\Theta$ is characterized by the coefficients:
\begin{equation} \label{eq:definition_theta}
	\theta_{\lambda, \mu} = \Theta[\lambda,\mu]= \dotproduct{ H \psi_{\lambda} }{ \psi_{\mu} }.
\end{equation}
The following result provides a good bound on their amplitude.
\begin{theorem}[Representation of blurring operator in wavelet bases \cite{escande2015sparse}] \label{thm:representation}

Let $f_{\lambda, \mu} = f\left( \mathrm{dist}(I_{\lambda},I_{\mu})  \right)$ and assume that:
\begin{itemize}
	\item Operator $H$ belongs to the class $\mathcal{A}(M,f)$ (see Definition \ref{def:blurring_operators}).
	\item The mother wavelet is compactly supported with $M$ vanishing moments.
\end{itemize}

Then for all $\lambda = (j,m,e) \in \Lambda$ and $\mu = (k,n,e') \in \Lambda$, with $e,e'\neq 0$:
\begin{equation} \label{eq:decay}
  \abs{\theta_{\lambda, \mu} } \leq C_M 2^{-\left(M + \frac{d}{2} \right)\abs{j-k}} 2^{-\min(j,k)\left(M + d\right)} f_{\lambda,\mu},
\end{equation}
where $C_M$ is a constant that does not depend on $\lambda$ and $\mu$.
\end{theorem}
The coefficients of $\Theta$ decay exponentially with respect to the scale difference and also as a function of the distance between the two wavelets supports.

\subsection{Approximation in wavelet bases}

In order to get a representation of the operator in a finite dimensional setting, the wavelet representation can be truncated at scale $J$. Let $\Theta^{(J)}$ denote the infinite matrix defined by:
\begin{equation}
 \Theta^{(J)}[\lambda,\mu] = \left\{\begin{array}{ll}
                                      \theta_{\lambda,\mu} & \textrm{if } |\lambda|\leq J \textrm{ and } |\mu|\leq J, \\
                                      0 & \textrm{otherwise}.
                                     \end{array}\right.
\end{equation}
This matrix contains at most $N^2$ non-zero coefficients, where $N=1+\sum_{j=0}^{J-1}(2^d-1) 2^{dj}$ denotes the numbers of wavelets kept to represent functions. 
The operator $H^{(J)} = \Psi\Theta^{(J)} \Psi^*$ is an approximation of $H$. 
The following theorem is a variation of \cite{beylkin1991fast}. 
Loosely speaking, it states that $H^{(J)}$ can be well approximated by a matrix containing only $O(N)$ coefficients.
\begin{theorem}[Computation of blurring operators in wavelet bases \cite{escande2015sparse}] \label{thm:proof_thresh}
Set $0\leq \eta\leq \log_2(N)^{-(M+d)/d}$. 
Let $\Theta^{(J)}_\eta$ be the matrix obtained by zeroing all coefficients in $\Theta^{(J)}$ such that 
\begin{equation}
2^{-min(j,k)(M+d)} f_{\lambda, \mu} \leq \eta. 
\end{equation}
Define $H^{(J)}_\eta=\Psi \Theta^{(J)}_\eta \Psi^*$.
Under the same hypotheses as Theorem \eqref{thm:representation}, the number of coefficients needed to satisfy $\norm{H^{(J)} - H^{(J)}_{\eta}}_{2 \rightarrow 2} \leq \epsilon$ is bounded above by
\begin{equation}\label{eq:wavelet_quality}
  C'_M N \; \epsilon^{-\frac{d}{M}}
\end{equation}
where $C'_M>0$ is independent of $N$.
\end{theorem}

This theorem has important consequences for numerical analysis. 
It states that evaluations of $H$ can be obtained with an $\epsilon$ accuracy using only $O(N \epsilon^{-\frac{d}{M}})$ operations. Note that the smoothness $M$ of the kernel is handled automatically.
Of interest, let us mention that \cite{meyer2000wavelets,beylkin1991fast} proposed similar results under less stringent conditions. 
In particular, similar inequalities may still hold true if the kernel blows up on the diagonal. 

\subsection{Discretization}\label{subsec:discretization}

In the discrete setting, the above results can be exploited as follows. 
Given a matrix $H \in \R^{N\times N}$ that represents a discretized version of $H$, we perform the change of basis:
\begin{equation}\label{eq:matrixH}
 \Theta = \Psi^* H\Psi,
\end{equation}
where $\Psi\in \R^{N\times N}$ is the discrete isotropic separable wavelet transform. 
Similarly to the continuous setting, matrix $\Theta$ is essentially concentrated along the diagonals of the wavelet subbands (see Figure \ref{fig:illustration_compression}). 

Probably the main reason why this decomposition has very seldom been used in practice is that it is very computationally demanding.
To compute the whole set of coefficients $(\langle H \psi_\lambda , \psi_\mu\rangle)_{\lambda,\mu}$, one has to apply $H$ to each of the $N$ discrete wavelets and then compute a decomposition of the blurred wavelet. This requires $O(N^3)$ operations, since blurring a wavelet is a matrix-vector product which costs $O(N^2)$ operations. This computational burden is intractable for large signals. We will see that it becomes tractable for \emph{convolution operators} in section \ref{subsec:convolutions}.

\subsection{Illustration}

In order to illustrate the various results provided so far, let us consider an operator acting on 1D signals, with kernel defined by 
\begin{equation}
K(x,y) = \frac{1}{\sigma(y) \sqrt{2\pi}} \exp{ \left( - \frac{(x-y)^2}{2\sigma^2(y)} \right)}, 
\end{equation}
where $\sigma(y) = 4 + 10 y$. All PSFs are Gaussians with a variance that increases linearly. The matrix is displayed in linear scale (resp. log scale) in Figure \ref{fig:illustration_compression} top-left (resp. top-right).
The wavelet representation of the matrix is displayed in linear scale (resp. log scale) in Figure \ref{fig:illustration_compression} bottom-left (resp. bottom-right). As can be seen, the matrix expressed in the wavelet basis is sparser than in the space domain. It has a particular banded structure captured by Theorem \ref{thm:representation}.
\begin{figure}[htp]
\centering
\begin{subfigure}[b]{0.45\textwidth}
	\includegraphics[width=\textwidth]{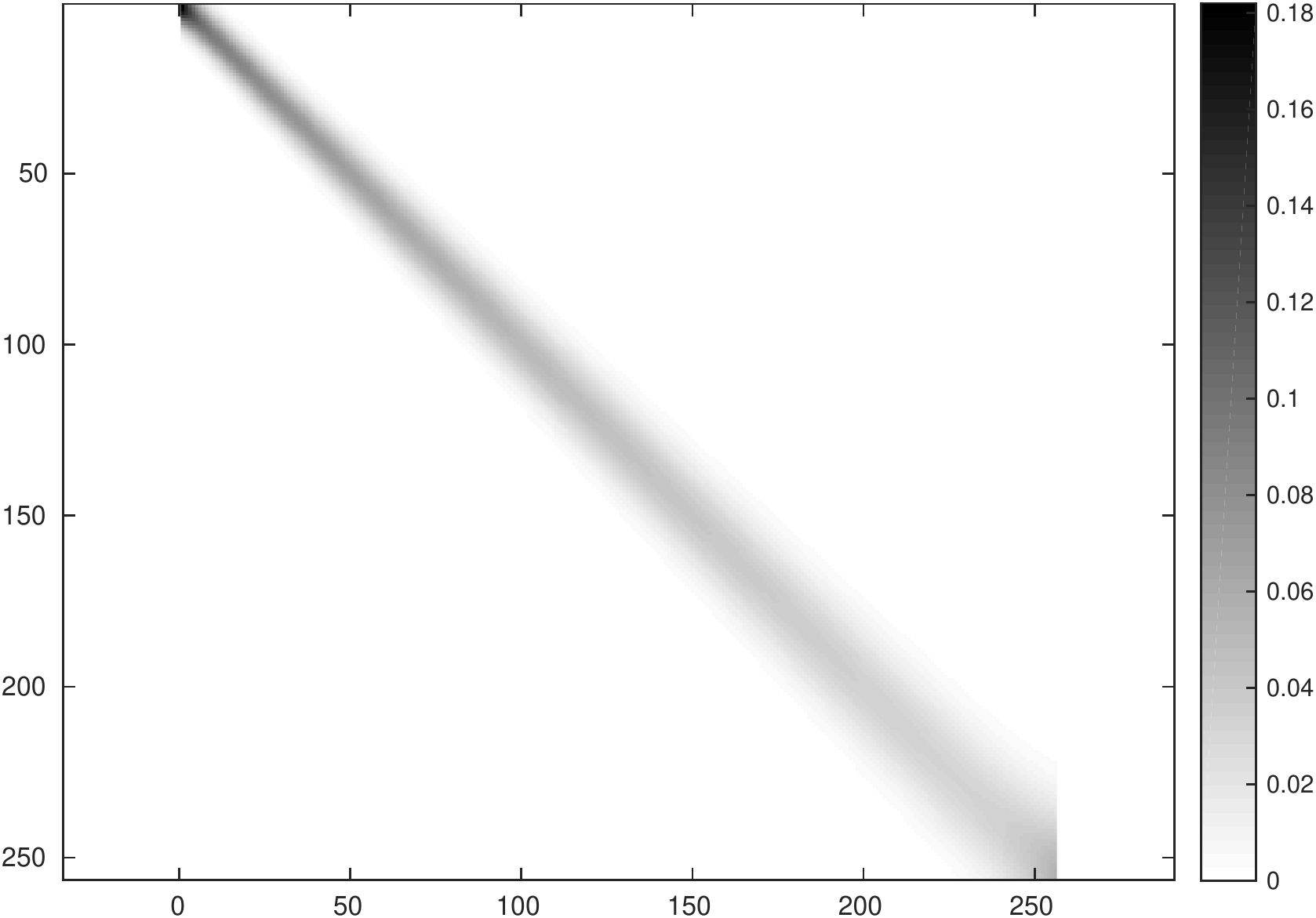}
\end{subfigure}
\quad
\begin{subfigure}[b]{0.45\textwidth} 
	\includegraphics[width=\textwidth]{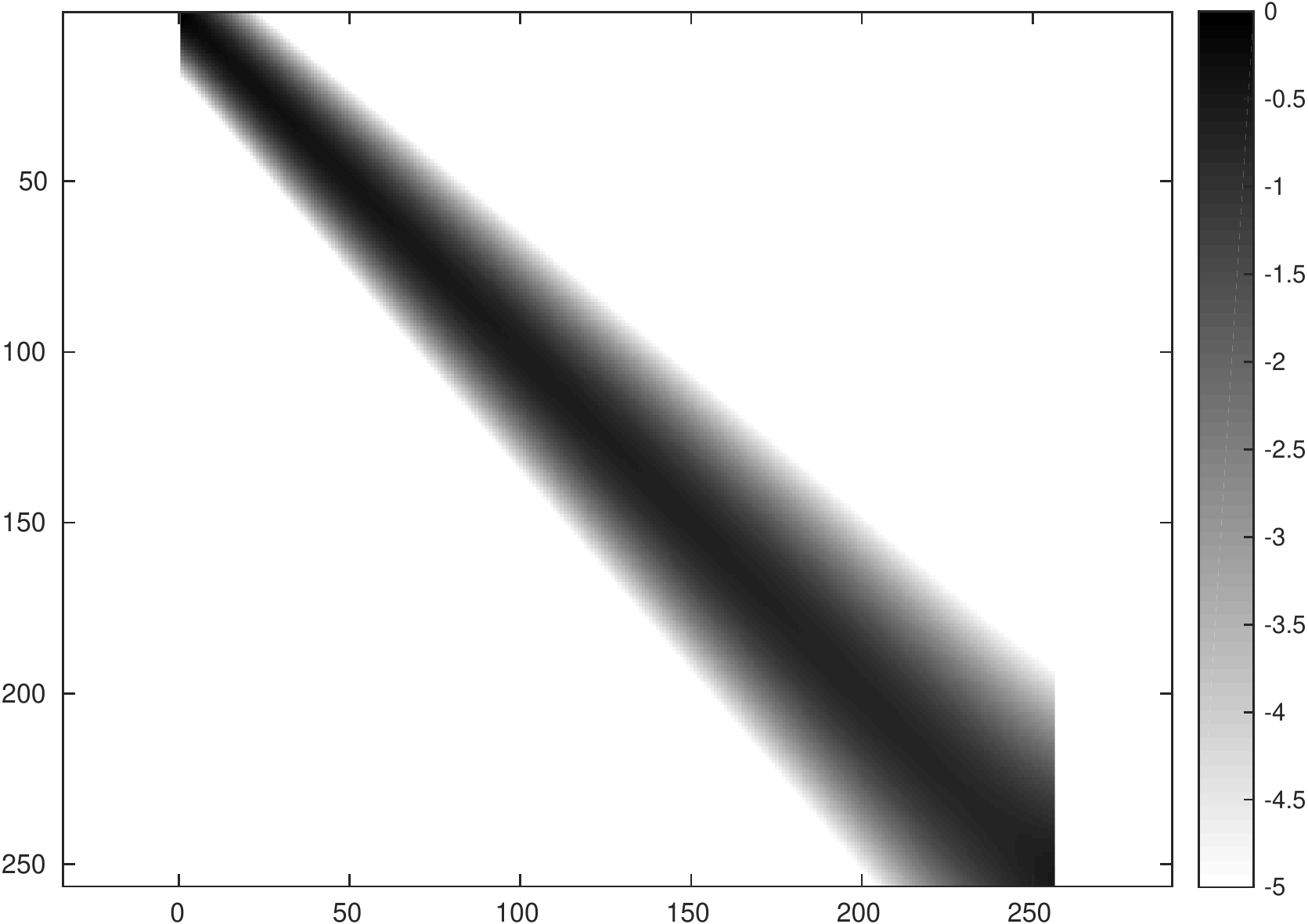}
\end{subfigure}

\begin{subfigure}[b]{0.45\textwidth}
	\includegraphics[width=\textwidth]{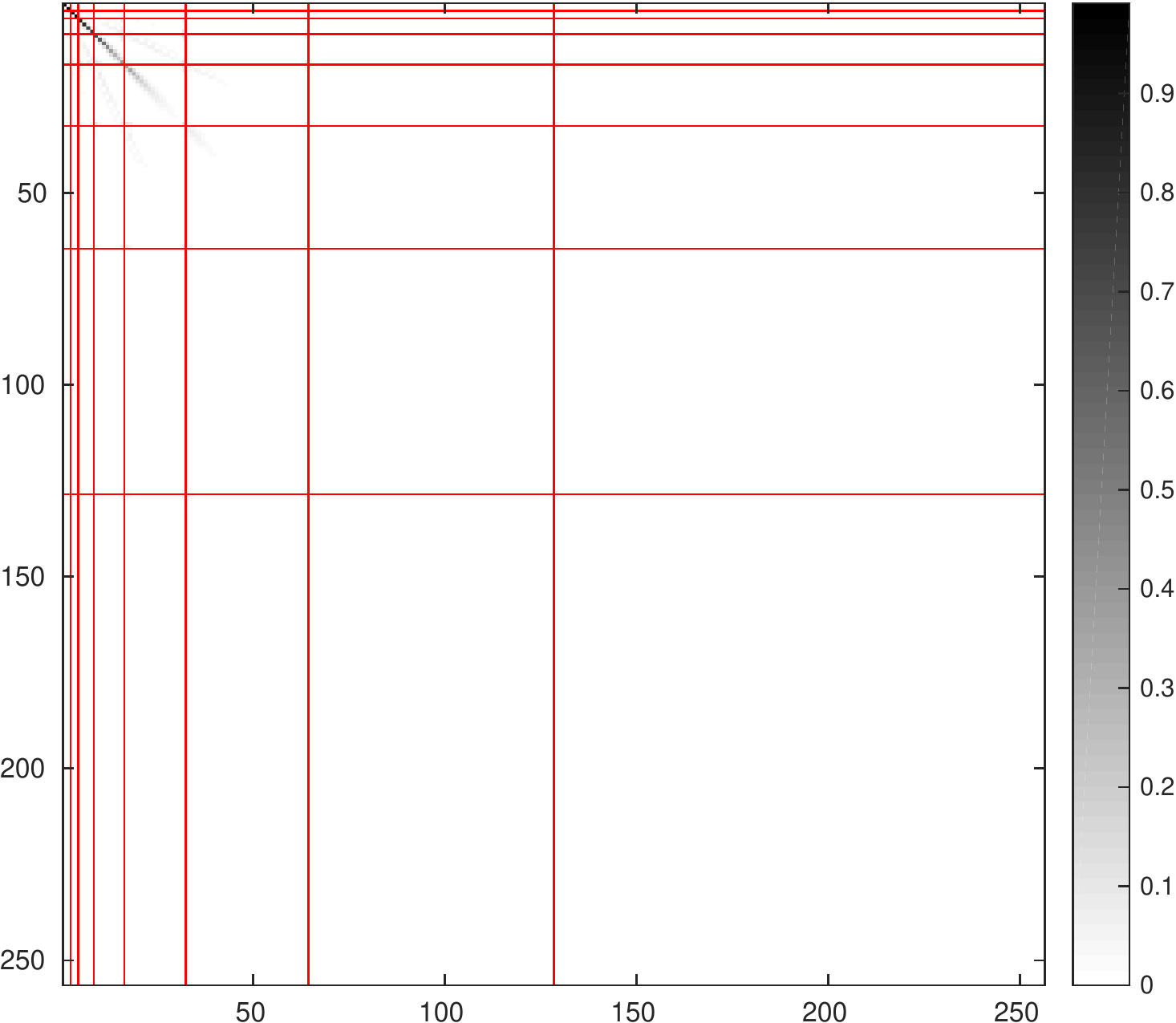}
\end{subfigure}
\quad
\begin{subfigure}[b]{0.45\textwidth} 
	\includegraphics[width=\textwidth]{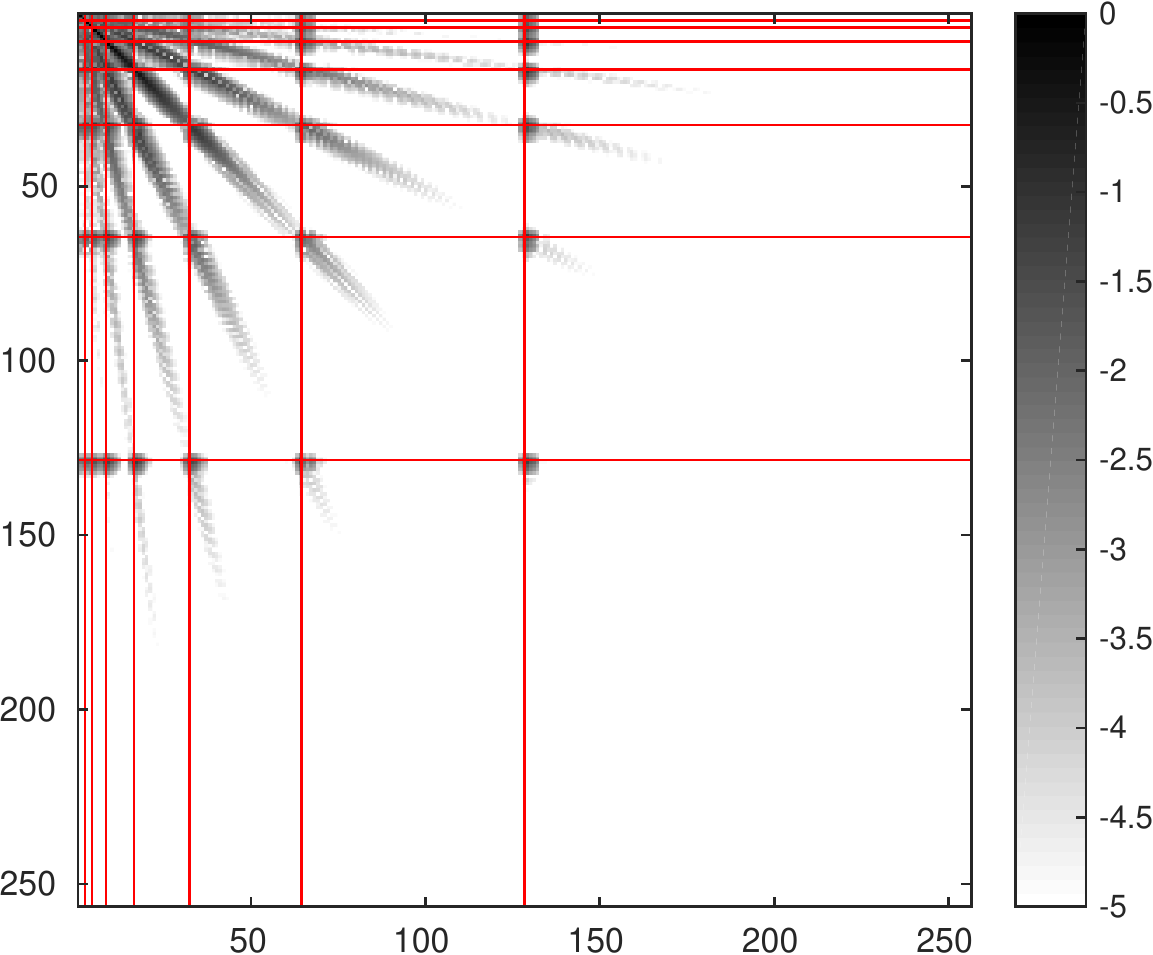}
\end{subfigure}
	\caption{An illustration of the compression of a spatially varying blur in the wavelet domain. Top-left: $H$. Top-right: $H$ in $\log_{10}$-scale. Bottom-left: $\Theta$ obtained using Daubechies wavelets with 10 vanishing moments and a decomposition level $J = 7$. Bottom-right: $\Theta$ in $\log_{10}$-scale.} \label{fig:illustration_compression}
\end{figure}

\subsection{Decomposition of convolutions}
\label{subsec:convolutions}

From now on, we assume that $H$ is a convolution with a kernel $h$. 
The results below hold both in the continuous and the discrete setting. 
We establish them in the discrete setting to ease the implementation.
Matrix $\Theta$ can be decomposed into its wavelet subbands:
\begin{equation}
	\Theta = \left( \Theta_{j,k}^{e,e'} \right)_{j,k}, \quad \text{with } \Theta_{j,k}^{e,e'} = \left( \dotproduct{H \psi_{j,m}^e}{\psi_{k,n}^{e'}}\right)_{m \in \mathcal{T}_j, n \in \mathcal{T}_k}.
\end{equation}
For instance, on 1D signals with $J = 2$, matrix $\Theta$ can be decomposed as shown in Figure \ref{fig:structure_convolution}, left. Let us now describe the specific structure of the subbands $\Theta_{j,k}^{e,e'}$ for convolutions. 
We will need the following definitions.

\begin{definition}[Translation operator]
Let $a\in \R^N$ denote a $d$-dimensional image and $m\in \Z^d$ denote a shift. 
The translated image $b=\tau_m(a)$ is defined for all $i_1,\hdots,i_d$ by:
\begin{equation}
b[i_1,\hdots,i_d] = a[i_1-m_1,\hdots,i_d-m_d]
\end{equation}
with circular boundary conditions.
\end{definition}

\begin{definition}[Rectangular circulant matrices]
	Let $A\in \R^{2^j \times 2^k}$ denote a rectangular matrix. It is called circulant if and only if:
	\begin{itemize}
		\item When $k \geq j$ : 	there exists $a \in \R^{2^k}$, such that, for all $0 \leq l \leq 2^j-1$, 
			\begin{equation*}
				A[l,:] = \tau_{2^{k-j} l}(a).
			\end{equation*}
		\item When $k < j$ : 	there exists  $a\in \R^{2^j}$, such that, for all $0 \leq l \leq 2^k-1$, 
			\begin{equation*}
				A[:,l] = \tau_{2^{j-k}l}(a).
			\end{equation*}
	\end{itemize}		
	As an example a $4 \times 8$ circulant matrix is of the form:
		\begin{equation*}
			A = \begin{pmatrix}
				\color{red}{a_1} 	& \color{blue}{a_2} & a_3 & a_4 & a_5 & a_6 & a_7 & a_8 \\
				a_7 & a_8 & \color{red}{a_1} & \color{blue}{a_2} & a_3 & a_4 & a_5 & a_6 \\
				a_5 & a_6 & a_7 & a_8 & \color{red}{a_1} & \color{blue}{a_2} & a_3 & a_4 \\
				a_3 & a_4 & a_5 & a_6 & a_7 & a_8 & \color{red}{a_1} & \color{blue}{a_2}
			\end{pmatrix}.
		\end{equation*}
	
\end{definition}
\begin{figure}[htp]
\centering
\begin{subfigure}[b]{0.31\textwidth}
	\includegraphics[width=\textwidth]{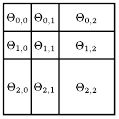}
\end{subfigure}
\quad
\begin{subfigure}[b]{0.35\textwidth} 
	\includegraphics[width=\textwidth]{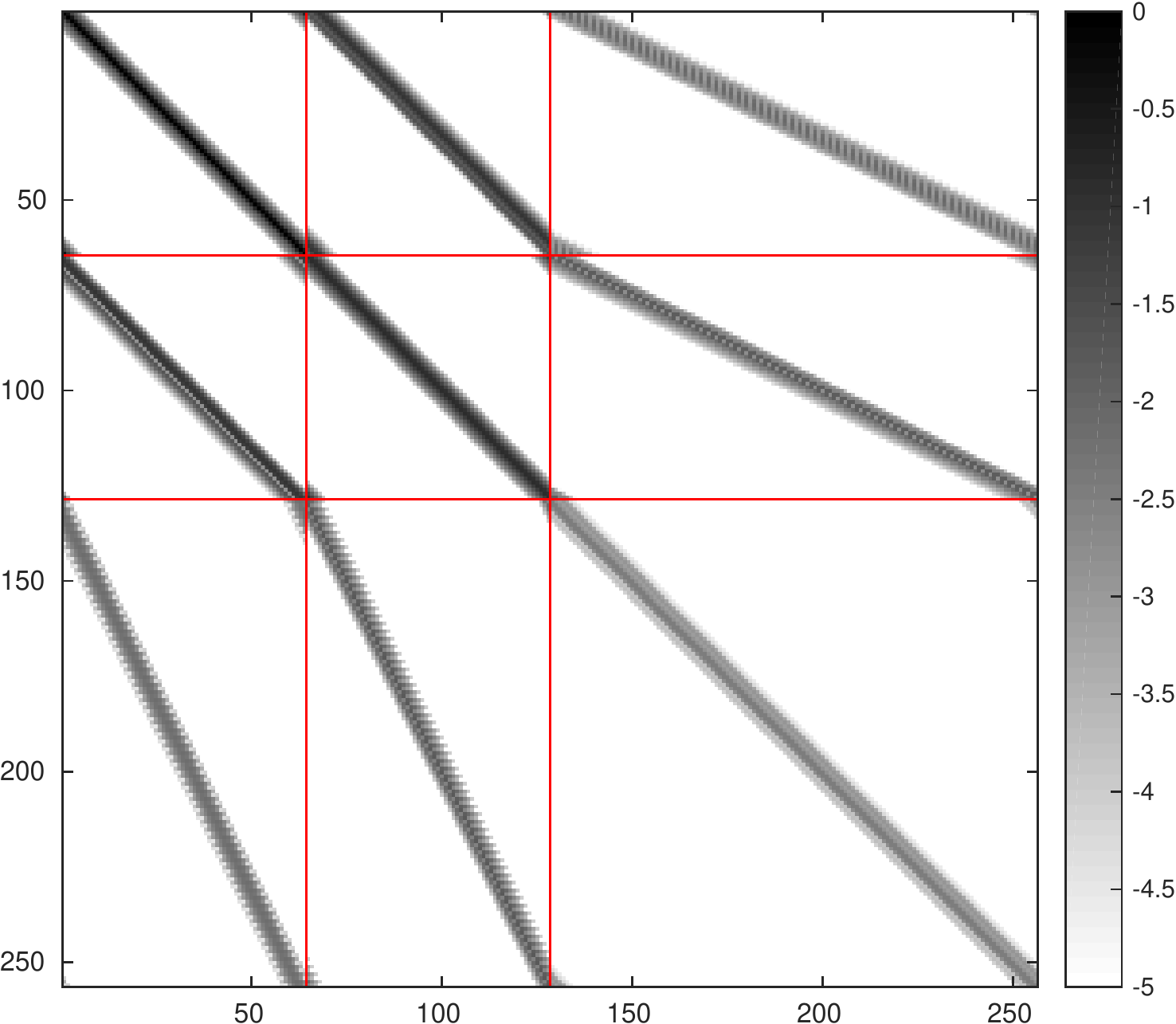}
\end{subfigure}
	\caption{Left: structure of $\Theta$. Right: $\Theta$ in $\log_{10}$-scale when $H$ is a convolution with a Gaussian kernel.}  \label{fig:structure_convolution}
\end{figure}

Theorem \ref{thm:theta_convolution} states that all the wavelet subbands of $\Theta$ are circulant for convolution operators. This is illustrated on Figure \ref{fig:structure_convolution}, right.

\begin{theorem}[Circulant structure of $\Theta$] \label{thm:theta_convolution}
Let $H$ be a convolution matrix and define $\Theta=\Psi^*H\Psi$ be its wavelet representation. 
Then, for all $j,k \in [0,J]$ and $e, e' \in \set{0,1}^d$, the submatrices $\Theta_{j,k}^{e,e'}$ are circulant.
\end{theorem}

\begin{proof}
We only treat the case $j \geq k$, since the case $j<k$ is similar. 
Let $a\in \R^{2^j}$ be defined by:
\begin{equation}
 a[m]= \dotproduct{h\star \psi_{j,m}^e}{ \psi_{k,0}^{e'}}.
\end{equation}

We have:
	\begin{align}
		\dotproduct{h\star \psi_{j,m}^e}{ \psi_{k,n}^{e'}} &= 
		\dotproduct{h\star \psi_{j,m}^e}{ \tau_{2^k n} \left( \psi_{k,0}^{e'}\right) } \\
		&= \dotproduct{\tau_{-2^k n} \left( h\star \psi_{j,m}^e \right)}{ \psi_{k,0}^{e'}} \\
		&= \dotproduct{h\star \tau_{-2^k n} \left( \psi_{j,m}^e \right)}{ \psi_{k,0}^{e'}} \\
		&= \dotproduct{h\star \psi_{j,m-2^{j-k}n}^e}{ \psi_{k,0}^{e'}} \\
		&= a[m-2^{j-k}n].
	\end{align}
	The submatrix $\Theta_{j,k}^{e,e'}$ is therefore circulant.
	In this list of identities, we only used the fact that the adjoint of $\tau_{2^k n}$ is $\tau_{-2^k n}$ and the fact that translations and convolution commute. 
\end{proof}

The main consequence of Theorem \ref{thm:theta_convolution} is that the computation of matrix $\Theta$ reduces to computing one column or one row of each matrix $\Theta_{j,k}^{e,e'}$. This can be achieved by computing $(2^d-1)J$ wavelet transforms (see Algorithm \ref{algo:btheta_conv}). The complexity of computing $\Theta$ therefore reduces to $\mathcal{O}\left( (2^d-1)J N \right)$ operations instead of $\mathcal{O}(N^3)$ operations for spatially varying operators.

\begin{algorithm}
\caption{An algorithm to compute $\Theta$ for convolution operator}
\label{algo:btheta_conv}
\begin{algorithmic}[1]
\State \textbf{input:} $h\in \R^N$, the convolution kernel of $H$.
\State \textbf{output:} $\Theta$, the wavelet representation of $H$
\For{$(j,e) \in [0,J] \times \set{0,1}^d$}
\State Compute the wavelet $\psi_{\lambda}$ with $\lambda = (j,e, 0)$.
\State Compute the blurred wavelets $H \psi_{\lambda}$ and $H^* \psi_{\lambda}$.
\State Compute $\left(\dotproduct{H \psi_{\lambda}}{\psi_{\mu}}\right)_{\mu}$ using one forward wavelet transform.
\State Compute $\left(\dotproduct{H^* \psi_{\lambda}}{\psi_{\mu}}\right)_{\mu}$ using one forward wavelet transform.
  \For{$(k,e') \in [0,J] \times \set{0,1}^d$}
   		\If{ $k \geq j$}
   			\State $\Theta_{j,k}^{e,e'}$ is the circulant matrix with column: $\left(\dotproduct{H \psi_{\lambda}}{\psi_{k,n}^{e'}}\right)_{n}$
   		\Else
   			\State $\Theta_{j,k}^{e,e'}$ is the circulant matrix with row: $\left(\dotproduct{H^* \psi_{\lambda}}{\psi_{k,n}^{e'}}\right)_{n} = \left(\dotproduct{\psi_{\lambda}}{H \psi_{k,n}^{e'}}\right)_{n} $
   		\EndIf
   	\EndFor
  \EndFor
\end{algorithmic}
\end{algorithm}

\subsection{Thresholding strategies}
\label{subsec:thresholdstrategies}

Theorem \ref{thm:proof_thresh} ensures that one can construct good sparse approximations of $\Theta$. 
However, the thresholding strategy suggested by the theorem turns out to be impractical.

In this section, we propose efficient thresholding strategies. 
Most of the proposed ideas come from our recent work \cite{escande2015sparse}, and we refer to this paper for more details. 
The algorithm specific to convolution operators is new.

Let us define the operator norm:
\begin{equation}
\|H\|_{X\to Y} = \sup_{\|u\|_X\leq 1} \|Hu\|_Y,
\end{equation}
where $\|\cdot\|_X$ and $\|\cdot\|_Y$ denote two norms on $\R^N$.
A natural way to obtain a $K$-sparse approximation $\Theta_K$ of $\Theta$ consists of finding the minimizer of:
\begin{equation}\label{eq:thresholding}
 \min_{\Theta_K, \ K\textrm{-sparse}} \|H_K - H\|_{X\to Y},
\end{equation}
where $H_K=\Psi \Theta_K \Psi^*$.
The most naive thresholding strategy consists of constructing a matrix $\Theta_K$ such that:
\begin{equation}\label{eq:naivethresholding}
	\Theta_K[\lambda, \mu] = \left\{ \begin{aligned}
	\Theta[\lambda, \mu] &\quad \text{if $|\Theta[\lambda, \mu]|$ is among the $K$ largest values of $|\Theta|$}, \\
	0 	& \quad \text{otherwise}.
	\end{aligned} \right.
\end{equation}
This thresholding strategy can be understood as the solution of the minimization problem \eqref{eq:thresholding}, by setting 
$\|\cdot\|_X=\|\cdot\|_1$ and $\|\cdot\|_Y=\|\cdot\|_\infty$.

The $\ell^1$-norm of the wavelet coefficients is not adapted to the description of images. 
Natural images are often modeled as elements of Besov spaces or the space of bounded variation functions \cite{petrushev1999nonlinear,aubert2006mathematical}. 
These spaces can be characterized by the decay of their wavelet coefficients \cite{cohen2003numerical} across subbands. 
This motivates to set $\norm{ \cdot}_X = \norm{ \Sigma \cdot}_1$ where $\Sigma = \textrm{diag}(\sigma) \in \R^{N\times N}$ is a diagonal matrix and where $\sigma\in \R^N$ is constant by levels. 
The thresholding strategy using this metric can be expressed as the minimization problem:
\begin{equation}
\min_{\Theta_K,\   K\text{- sparse}} \sup_{\norm{\Sigma x}_1 \leq 1} \norm{ (\Theta - \Theta_K) x}_{\infty}.
\end{equation}
Its solution is given in closed-form by:
\begin{equation}\label{eq:cleverthreshold}
	\Theta_K[\lambda, \mu] = \left\{ \begin{aligned}
				  \Theta[\lambda, \mu]  &\quad \text{if $\left|\sigma_{\mu} \Theta[\lambda, \mu] \right|$ is among the $K$ largest values of $|\Theta \Sigma|$,} \\
				  0 	& \quad \text{otherwise}.
			\end{aligned} \right.
\end{equation}
The weights $\sigma_\mu$ must be adapted to the class of images to recover. 
In practice we found that setting $\sigma_\mu = 2^{-k}$ for $\mu = (k,e',n)$ is a good choice. 
These weights can also be trained from a set of images belonging the class of interest. 
Finally let us mention that we also proposed greedy algorithms when setting $\|\cdot\|_Y=\|\cdot\|_2$ in \cite{escande2015sparse}. 
In practice, it turns out that both approaches yield similar results.

We illustrate the importance of the thresholding strategy in section \ref{sec:num_threhsold}, Figure \ref{fig:deconv_compare_thresh}. 

%% file: Body/20-methods.tex
\section{On the design of preconditioners}\label{sec:optimization}

Let $P\in \R^{N\times N}$ denote a Symmetric Positive Definite (SPD) matrix. There are two equivalent ways to understand preconditioning: one is based on a change of variable, while the other is based on a metric change. 

For the change of variable, let $z$ be defined by $x=P^{1/2}z$. A solution $x^*$ of problem \eqref{eq:l1problemThetaApprox} reads $x^*=P^{1/2}z^*$, where $z^*$ is a solution of:
\begin{equation}
\min_{z\in \R^N} \frac{1}{2}\|\Theta_K P^{1/2} z - x_0 \|_2^2 + \|P^{1/2} z\|_{1,w}. \label{eq:l1problemThetaApproxPrecond}
\end{equation}
The convergence rate of iterative methods applied to \eqref{eq:l1problemThetaApproxPrecond} is now driven by the properties of matrix $\Theta_K P^{1/2}$ instead of $\Theta_K$. 
By choosing $P$ adequately, one can expect to significantly accelerate convergence rates.
For instance, if $\Theta_K$ is invertible and $P^{1/2}=\Theta_K^{-1}$, one iteration of a proximal gradient descent provides the exact solution of the problem. 

For the metric change, the idea is to define a new scalar product defined by 
\begin{equation}
\langle x,y\rangle_P = \langle Px,y\rangle,
\end{equation}
and to consider the associated norm $\|x\|_P=\sqrt{\langle Px,x\rangle}$. 
By doing so, the gradient and proximal operators are modified, which leads to different dynamics. 
The preconditioned FISTA algorithm is given in Algorithm \ref{alg:FISTAP}.
\begin{algorithm}
\caption{Preconditioned accelerated proximal gradient descent}
\label{alg:FISTAP}
\begin{algorithmic}[1]
\State \textbf{input:} Initial guess $x^{(0)}=y^{(1)}$, $\tau=1/\|\Theta_K^*\Theta_K P^{-1}\|_2$ and $Nit$.
\For{$k=1$ to $Nit$}
\State Compute $\nabla F(y^{(k)}) = \Theta_K^*(\Theta_K y^{(k)} -u_0 )$. 
\State $x^{(k)} = \prox_{\tau G }^P \left( y^{(k)} - \tau P^{-1} \nabla F(y^{(k)}) \right)$.
\State  $y^{(k+1)} = x^{(k)} + \frac{k-1}{k+2} (x^{(k)} - x^{(k-1)})$. 
\EndFor
\end{algorithmic}
\end{algorithm}

In this algorithm
\begin{equation}\label{eq:defProxP}
 \prox_{\tau G}^P(z_0) = \argmin_{z\in \R^N} \frac{1}{2}\|z-z_0\|_P^2 + \tau G(z).
\end{equation}
Unfortunately, it is impossible to provide a closed-form expression of \eqref{eq:defProxP}, unless matrix $P$ has a very simple structure (e.g. diagonal). Finding an efficient preconditioner therefore requires: i) defining a structure for $P$ compatible with fast evaluations of the proximal operator \eqref{eq:defProxP} and ii) improving some ``properties'' of $\Theta_K P^{1/2}$ using this structure.


\subsection{What governs convergence rates?}

Good preconditioners are often heuristic. The following sentence is taken from a reference textbook about the resolution of linear systems by Y. Saad \cite{saad2003iterative}:
``Finding a good preconditioner to solve a given sparse linear system is often viewed as a combination of art and science. Theoretical results are rare and some methods work surpisingly well, often despite expectation.''
In what follows, we will first show that existing convergence results are indeed of little help. We then provide two simple diagonal preconditioners.

Let us look at the convergence rate of Algorithm \ref{alg:FISTA} applied to problem \eqref{eq:l1problemThetaApproxPrecond}. 
The following theorem appears in \cite{nesterov2013gradient,beck2009fast} for instance.
\begin{theorem}
Let $A=P^{-1/2}\Theta_K^*\Theta_K P^{-1/2}$ and set $L=\lambda_{\max}(A^*A)$. The iterates in Algorithm \ref{alg:FISTAP} satisfy:
\begin{equation}\label{eq:cvrate}
 E(x^{(k)}) - E(x^*) \leq L\|x-x_0\|_2^2 \cdot \min\left( \frac{1}{k^2} , \frac{1}{2}\left( \frac{\sqrt{\kappa(A) - 1}}{\sqrt{\kappa(A) + 1}} \right)^{2k} \right),
\end{equation}
where $\kappa(A)$ designs the condition number of $A$:
\begin{equation}
 \kappa(A) = \left\{\begin{array}{ll}
                     \sqrt{\frac{\lambda_{\max}(A^*A)}{\lambda_{\min}(A^*A)}} & \textrm{if } \lambda_{\min}(A^*A)>0, \\
                     +\infty & \textrm{otherwise}.
                    \end{array}\right.
\end{equation}
\end{theorem}

When dealing with ill-posed inverse problems, the condition number $\kappa(A)$ is huge or infinite and bound \eqref{eq:cvrate} therefore reduces to 
\begin{equation}
E(x^{(k)}) - E(x^*) \leq \frac{L\|x-x_0\|_2^2}{k^2},
\end{equation}
even for a very large number of iterations.
Unfortunately, this bound tells very little about which properties of $A$ characterize the convergence rate. Only the largest singular value of $A$ seems to matter. The rest of the spectrum does not appear, while it obviously plays a key role. 

Recently, more subtle results were proposed in \cite{tao2015local} for the specific $\ell^1-\ell^2$ problem and in \cite{liang2015activity} for a broad class of problems. These theoretical results were shown to fit some experiments very well, contrarily to Theorem \ref{eq:cvrate}. Let us state a typical result.
\begin{theorem}\label{thm:jalal}
Assume that problem \eqref{eq:l1problemThetaApproxPrecond} admits a unique minimizer $x^*$.
Let $S^*=\mathrm{supp}(x^*)$ denote the solution's support. 
Then:
\begin{itemize}
 \item The sequence $(x^{(k)})_{k\in \N}$ converges to $x^*$.
 \item There exists an iteration number $k^*$ such that, for $k\geq k^*$, $\mathrm{supp}(x^{(k)})=\mathrm{supp}(x^*)$.
 \item If in addition 
 \begin{equation}\label{eq:conditionrestricted}
\langle A x, A x \rangle \geq \alpha \|x\|_2^2 , \ \forall x \textrm{ s.t. } \mathrm{supp}(x)\subseteq S^*,
 \end{equation}
then the sequence of iterates $(x^{(k)})_{k\in \N}$ converges linearly to $x^*$: there exists $0\leq \rho<1$ s.t. 
\begin{equation}
\|x^{(k)} - x^*\|_2 = O\left(\rho^k\right). 
\end{equation}
\end{itemize}
\end{theorem}

\begin{remark}
The uniqueness of a solution $x^*$ is not required if the algorithm converges. 
This can be ensured if the algorithm is slightly modified \cite{chambolle2014convergence}. 
\end{remark}

The main consequence of Theorem \eqref{thm:jalal} is that good preconditioners should depend on the support $S^*$ of the solution. Obviously, this support is unknown at the start of the algorithm. Moreover, for compact operators, condition \eqref{eq:conditionrestricted} is hardly satisfied. Therefore - once again - Theorem \eqref{thm:jalal} seems to be of little help to find well founded preconditioners.

In this paper we therefore restrict our attention to two standard preconditioners: Jacobi and Sparse Approximate Inverses (SPAI) \cite{grote1997parallel,saad2003iterative}. The overall idea is to cluster the eigenvalues of $A^*A$.

\subsection{Jacobi preconditioner}

The Jacobi preconditioner is one of the most popular diagonal preconditioner, it consists of setting 
\begin{equation}\label{eq:Jacobi_Precond}
P = \max(\diag(\Theta_K^* \Theta_K), \epsilon),
\end{equation}
where $\epsilon$ is a small constant. The parameter $\epsilon$ guarantees the invertibility of $P$. 

The idea of this preconditioner is to make the Hessian matrix $P^{-1/2}\Theta_K^*\Theta_K P^{-1/2}$ ``close'' to the identity.
This preconditioner has a simple analytic expression and is known to perform well for diagonally dominant matrices. 
Blurring matrices expressed in the wavelet domain have a fast decay away from the diagonal, but are usually not diagonally dominant.
Moreover, the parameter $\epsilon$ has to be hand-tuned. 



\subsection{SPAI preconditioner}

The preconditioned gradient in Algorithm \ref{alg:FISTAP} involves matrix $P^{-1}\Theta_K^*\Theta_K$.
The idea of sparse approximate inverses is to cluster the eigenvalues of $P^{-1}\Theta_K^*\Theta_K$ around $1$. 
To improve the clustering, a possibility is to solve the following optimization problem:
\begin{equation} \label{eq:opt_diag}
\argmin_{P, \text{ diagonal}} \norm{ \mathrm{Id} - P^{-1}\Theta_K^*\Theta_K }_F^2,
\end{equation}
where $\|\cdot\|_F$ denotes the Frobenius norm.
This formulation is standard in the numerical analysis community and known as sparse approximate inverse (SPAI) \cite{grote1997parallel,saad2003iterative}.

\begin{lemma}
	Let $M=\Theta_K^* \Theta_K$. The set of solutions of \eqref{eq:opt_diag} reads:
	\begin{equation}\label{eq:SPAI_Precond}
		P[i,i] = \left\{ \begin{array}{l}
				  \frac{M^2[i,i]}{M[i,i] } \quad \textrm{if } M[i,i]\neq 0, \\
				  \textrm{an arbitrary positive value otherwise}.
		                 \end{array}\right.
	\end{equation}
\end{lemma}
\begin{proof}
First notice that problem \eqref{eq:opt_diag} can be rewritten as 
\begin{equation} \label{eq:opt_diag2}
\argmin_{P, \text{ diagonal}} \norm{ \mathrm{Id} - M P^{-1}}_F^2,
\end{equation}
by taking the transpose of the matrices, since $M$ is symmetric and $P$ diagonal.
The Karush-Kuhn-Tucker optimality conditions for problem \eqref{eq:opt_diag2} yield the existence of a Lagrange multiplier $\mu\in \R^{N\times N}$ such that:
\begin{equation}
 M(MP^{-1}-I) + \mu =0,
\end{equation}
with 
\begin{equation}
 \mu[i,j]=\left\{
\begin{array}{l}
 0 \quad \textrm{if } i= j,\\
 \textrm{an arbitrary value otherwise}.           
 \end{array} \right.
\end{equation}
Therefore, for all $i$,
\begin{equation}
 (M^2 P^{-1})[i,i] = M[i,i],
\end{equation}
which can be rewritten as 
\begin{equation}
 M^2[i,i] P^{-1}[i,i] = M[i,i],
\end{equation}
since $P$ is diagonal. If $M^2[i,i]=0$, then $M[i,i]=0$ since $M^2[i,i]$ is the squared norm of the $i$-th column of $M$. In that case, $P^{-1}[i,i]$ can take an arbitrary value. Otherwise $P^{-1}[i,i]= M[i,i] /M^2[i,i]$, finishing the proof.
\end{proof}

%% file: Body/30-numerical.tex
\section{Numerical experiments}\label{sec:numerical}

In this section we propose a set of numerical experiments to illustrate the proposed methodology and to compare its efficiency with respect to state-of-the-art approaches. 
The numerical experiments are performed on two $1024 \times 1024$ images with values rescaled in $[0,1]$, see Figure \ref{fig:original}. 
We also consider two different blurs, see Figure \ref{fig:PSFs}. 
The PSF in Figure \ref{fig:PSF_anisotiop_gauss} is an anisotropic 2D Gaussian with kernel defined for all $(t_1,t_2) \in [0,1]^2$ by 
\[
	k(t_1,t_2) = \left\{ \begin{array}{ll}
	                      \exp{ \left( - \frac{t_1^2}{2 \sigma^2} - \frac{t_2^2}{2 \sigma^2} \right)}  & \textrm{if } t_1\geq 0, \\
	                      \exp{ \left( - \frac{4t_1^2}{2 \sigma^2} - \frac{t_2^2}{2 \sigma^2} \right)}  & \textrm{otherwise},\\
	                     \end{array}\right.
\]
with $\sigma = 5$. This PSF is smooth, which is a favorable situation for our method, see Theorem \ref{thm:proof_thresh}.
The PSF in Figure \ref{fig:PSF_movement} is a simulation of motion blur. This PSF is probably one of the worst for the proposed technique since it is singular.
The PSF is generated from a set of $l = 5$ points drawn at random from a Gaussian distribution with standard deviation $\sigma_1 = 8$. 
Then the points are joined using a cubic spline and the resulting curve is blurred using a Gaussian kernel of standard deviation $\sigma_2=1$.

\begin{figure}[htp]
	\centering
	\begin{subfigure}[b]{0.35\textwidth}
		\includegraphics[width=\textwidth, trim=450 450 450 450, clip=true]{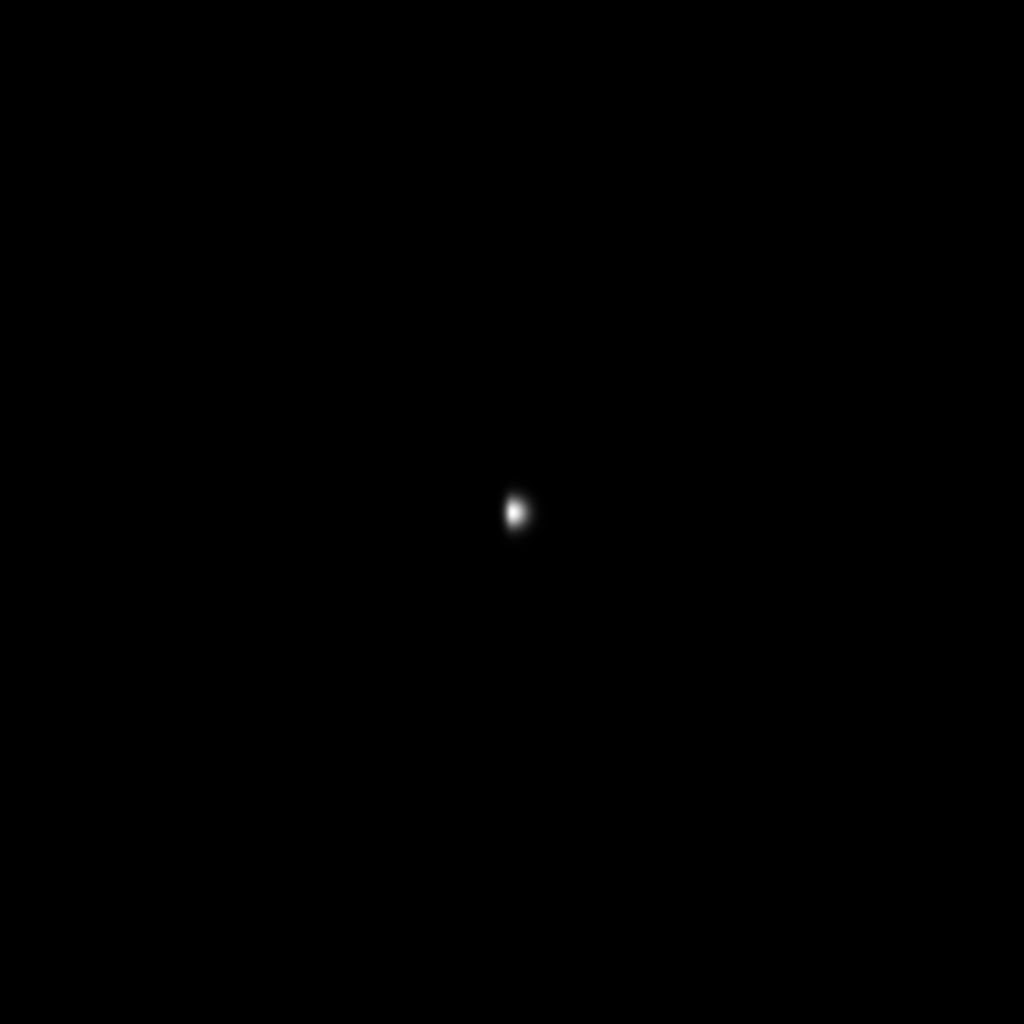}
		\caption{}\label{fig:PSF_anisotiop_gauss}
	\end{subfigure}
	\begin{subfigure}[b]{0.35\textwidth}
	\includegraphics[width=\textwidth, trim=450 450 450 450, clip=true]{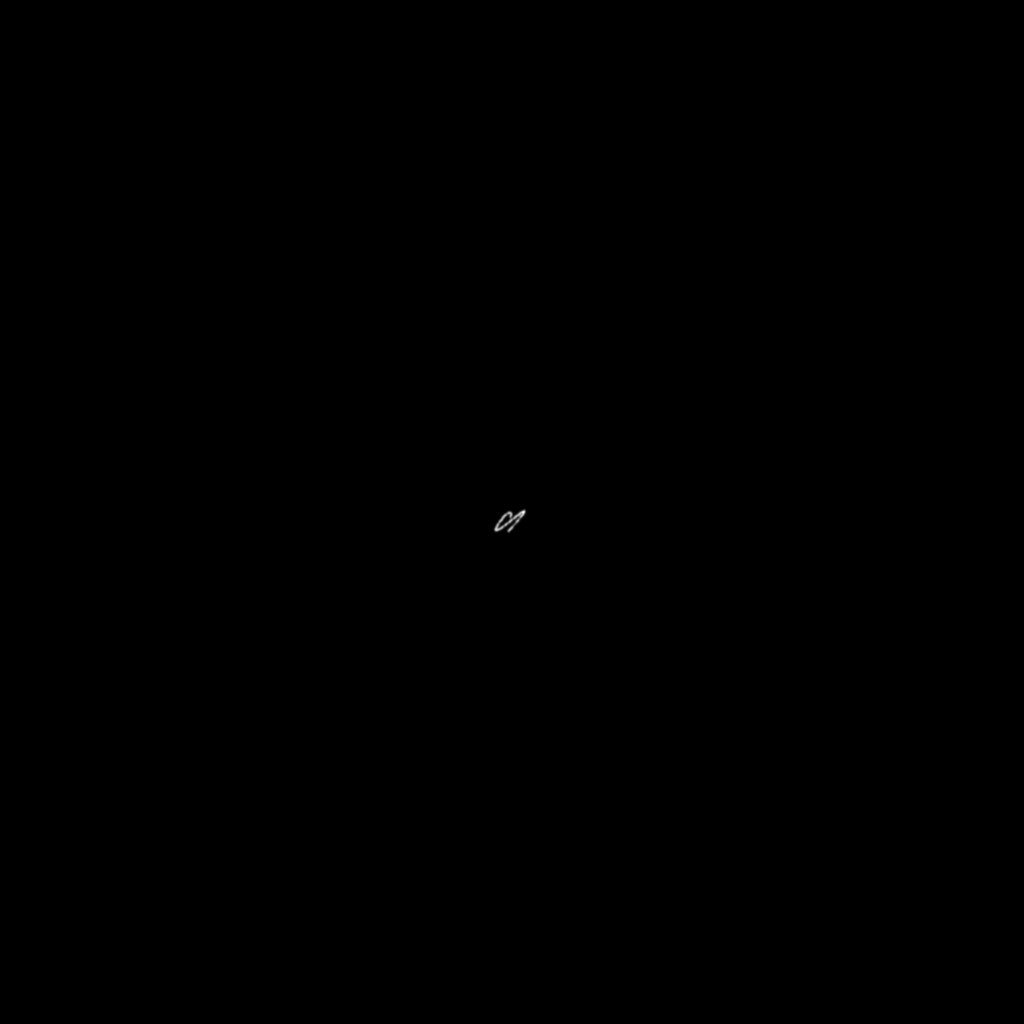}
	\caption{} \label{fig:PSF_movement}
	\end{subfigure}
	\caption{PSFs used in the paper. \eqref{fig:PSF_anisotiop_gauss} is a skewed Gaussian. \eqref{fig:PSF_movement} is a motion blur.} \label{fig:PSFs}
\end{figure}

All our numerical experiments are based on Symmlet 6 wavelets decomposed $J=6$ times. This choice offers a good compromise between computing times and visual quality of the results. The weights $w$ in function $G$ in \eqref{def:G} were defined by $w[i]=j(i)$, where $j(i)$ denotes the scale of the $i$-th wavelet coefficient. This choice was hand tuned so as to produce the best deblurring results.
The numerical experiments were performed on Matlab2014b on an Intel(R) Xeon(R) CPU E5-2680 v2 @ 2.80GHz with 200Gb RAM. Multicore was disabled by launching Matlab with:
\begin{lstlisting}[language=matlab]
>>  matlab -singleCompThread
\end{lstlisting}
For the experiments led on GPU, we use a NVIDIA Tesla K20c containing 2496 CUDA cores and 5GB internal memory.

Figures \ref{fig:deconv_book} and \ref{fig:deconv_confocal} display two typical deconvolution results using this approach.

\begin{figure}\centering
\begin{subfigure}[b]{0.9\textwidth}
\begin{tikzpicture}[zoomboxarray]
    \node [image node] { \includegraphics[width=0.45\textwidth]{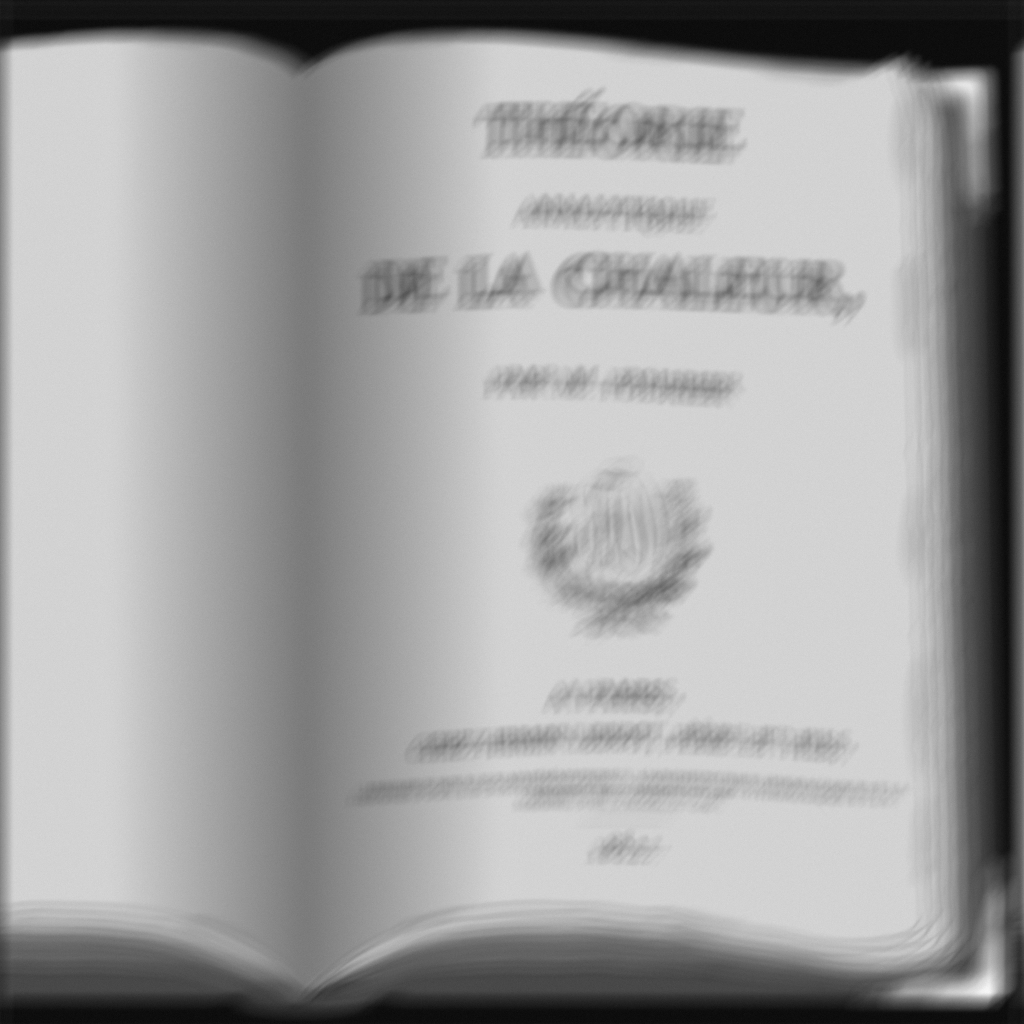} };
    \zoombox[color code=red]{0.905,0.1}
    \zoombox[magnification=2]{0.6,0.47}
    \zoombox[color code=blue,magnification=2]{0.55,0.8}
	\zoombox[color code=green,magnification=2]{0.6,0.15}
\end{tikzpicture}
\end{subfigure}

\begin{subfigure}[b]{0.9\textwidth}
\begin{tikzpicture}[zoomboxarray]
    \node [image node] { \includegraphics[width=0.45\textwidth]{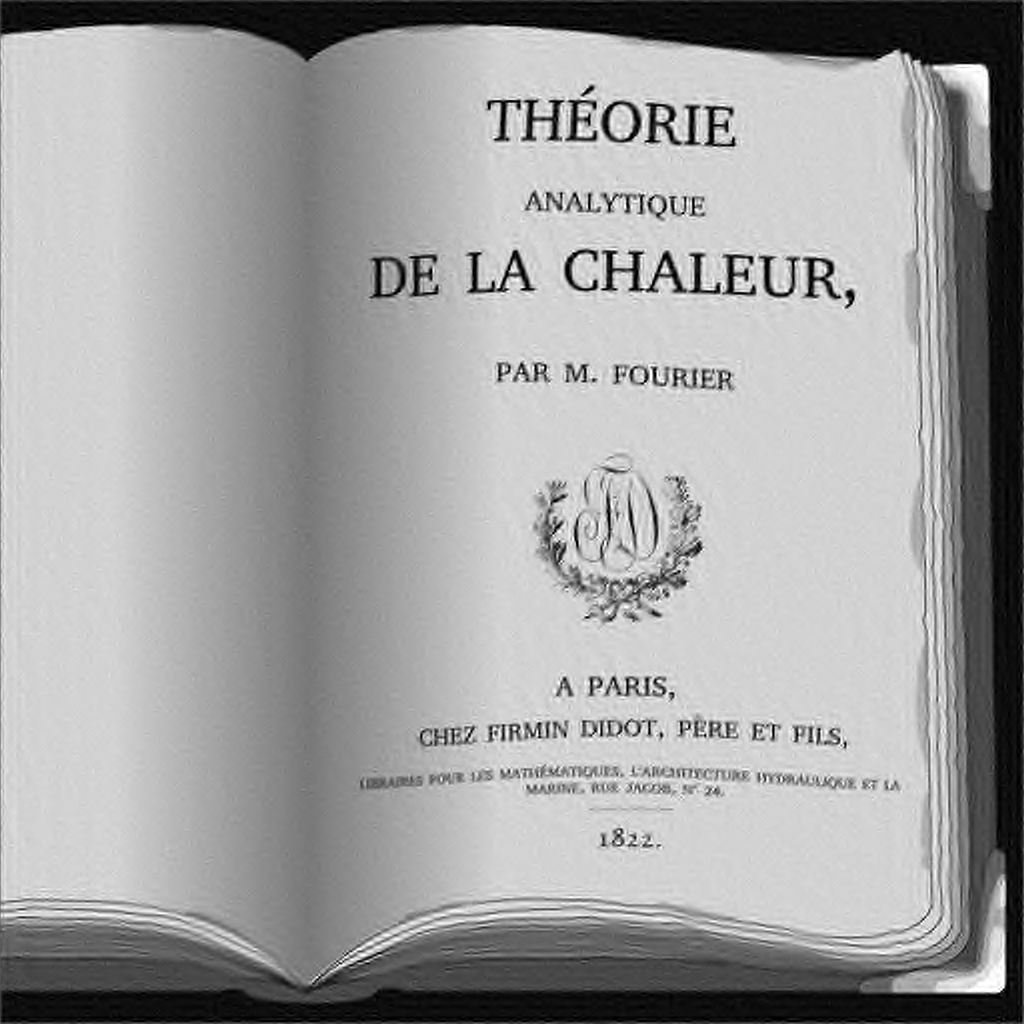} };
    \zoombox[color code=red]{0.905,0.1}
    \zoombox[magnification=2]{0.6,0.47}
    \zoombox[color code=blue,magnification=2]{0.55,0.8}
	\zoombox[color code=green,magnification=2]{0.6,0.15}
\end{tikzpicture}
\end{subfigure}
\caption{A deconvolution example. The book image in Figure \ref{fig:original} is blurred with the motion blur Figure \ref{fig:PSF_movement} and degraded with a noise level of $5.10^{-3}$. The pSNR of the degraded image (on top) is 17.85dB.  Problem \eqref{eq:l1problem} is solved using the exact operator, $\lambda = 10^{-4}$, 500 iterations and Symmlet 6 wavelets decomposed 6 times. The pSNR of the restored image is 24.14dB.} \label{fig:deconv_book}
\end{figure}


\begin{figure}\centering
\begin{subfigure}[b]{0.9\textwidth}
\begin{tikzpicture}[zoomboxarray]
    \node [image node] { \includegraphics[width=0.45\textwidth]{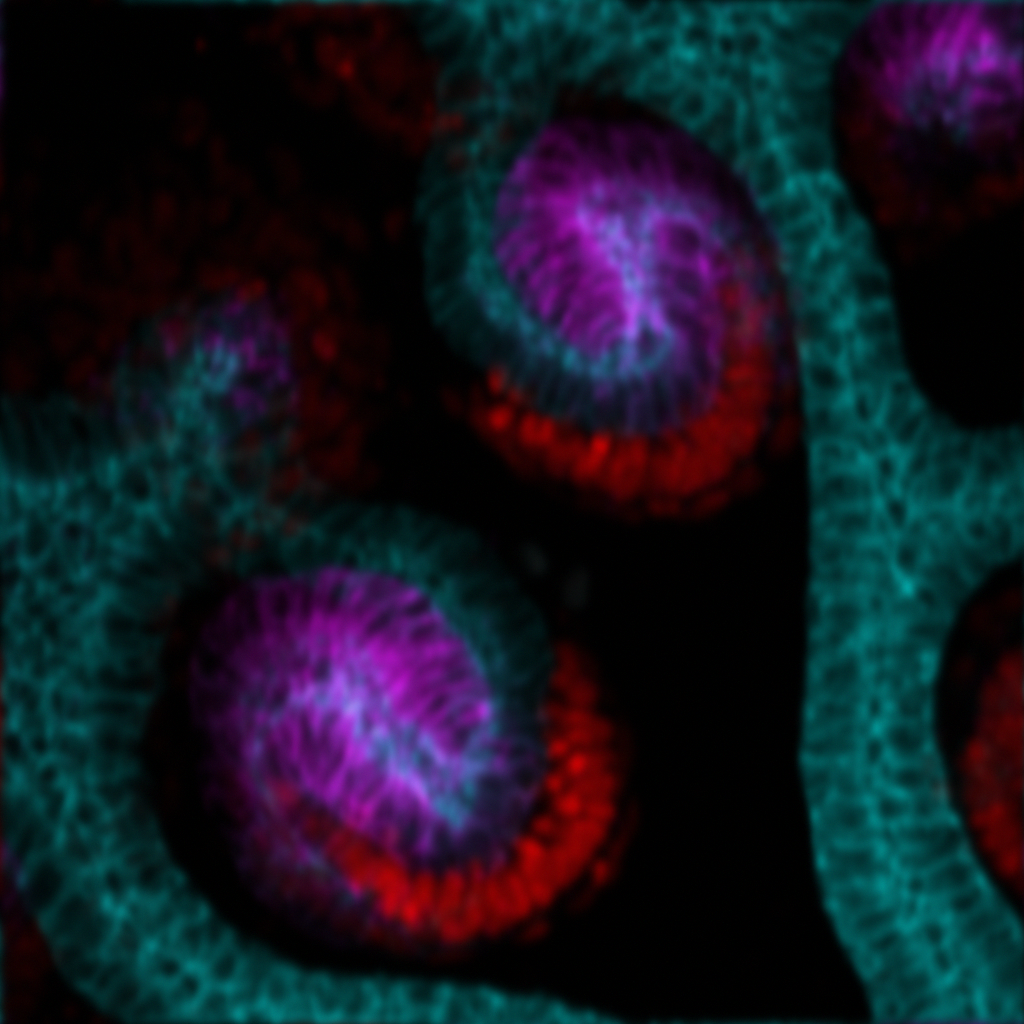} };
    \zoombox[color code=red]{0.9,0.1}
    \zoombox[magnification=2]{0.875,0.5}
    \zoombox[color code=blue,magnification=2]{0.5,0.8}
	\zoombox[color code=green,magnification=2]{0.4,0.2}
\end{tikzpicture}
\end{subfigure}

\begin{subfigure}[b]{0.9\textwidth}
\begin{tikzpicture}[zoomboxarray]
    \node [image node] { \includegraphics[width=0.45\textwidth]{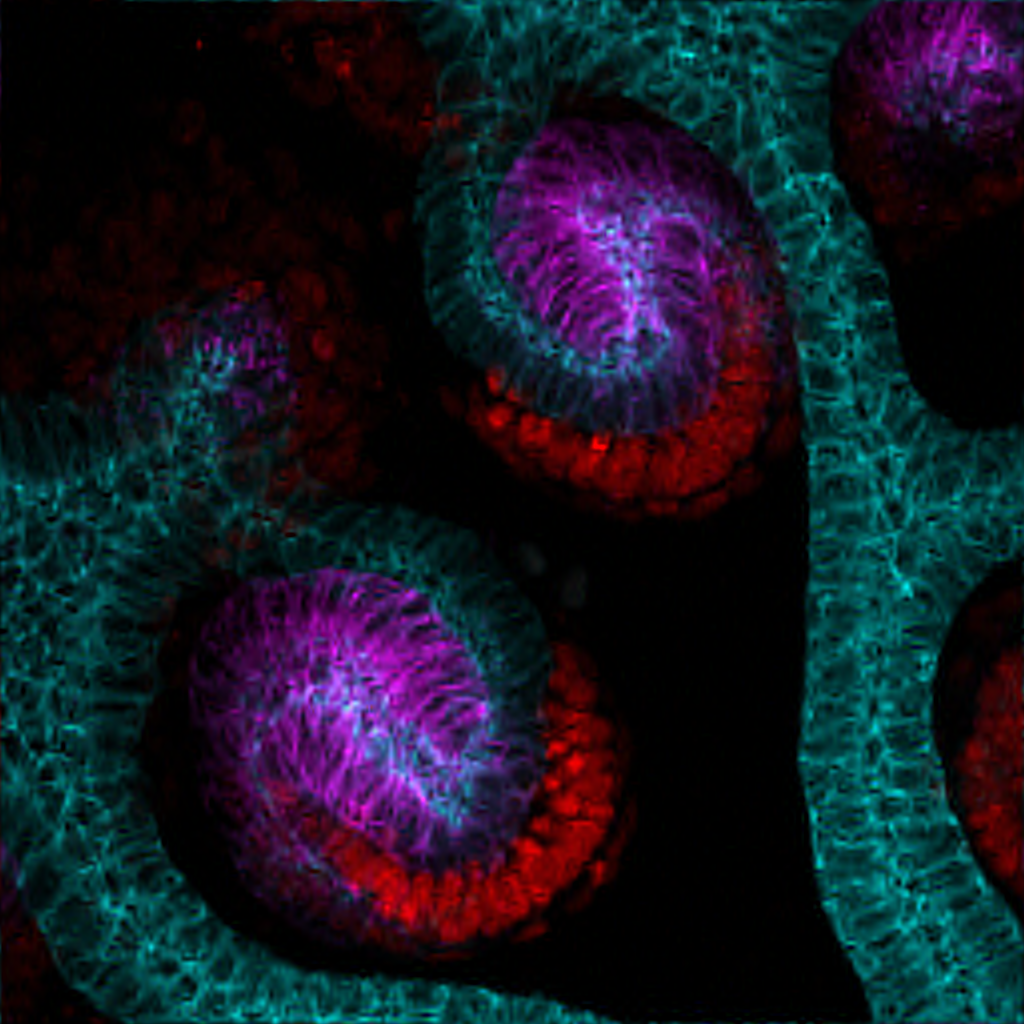} };
    \zoombox[color code=red]{0.9,0.1}
    \zoombox[magnification=2]{0.875,0.5}
    \zoombox[color code=blue,magnification=2]{0.5,0.8}
	\zoombox[color code=green,magnification=2]{0.4,0.2}
\end{tikzpicture}
\end{subfigure}
\caption{A deconvolution example. The confocal image Figure \ref{fig:original} has been blurred with blur Figure \ref{fig:PSF_anisotiop_gauss} and degraded with a noise level of $5.10^{-3}$. The pSNR of the degraded image (on top) is 23.94dB.  Problem \eqref{eq:l1problem} is solved using the exact operator, $\lambda = 10^{-4}$, 500 iterations and Symmlet 6 wavelets decomposed 6 times. The pSNR of the restored image is 26.33dB.} \label{fig:deconv_confocal} 
\end{figure}

\subsection{On the role of thresholding strategies} \label{sec:num_threhsold}

We first illustrate the influence of the thresholding strategy discussed in Section \ref{subsec:thresholdstrategies}. We construct two matrices having the same number of coefficients but built using two different thresholding rules: the naive thresholding given in equation \eqref{eq:naivethresholding} and the weighted thresholding given in equation \eqref{eq:cleverthreshold}.
Figure \ref{fig:deconv_compare_thresh} displays the images restored with each of these two matrices. 
It is clear that the weighted thresholding strategy significantly outperforms the simple one: it produces less artefacts and a higher pSNR. 
In all the following experiments, this thresholding scheme will be used.

\begin{figure}\centering
\begin{subfigure}[b]{0.9\textwidth}
\begin{tikzpicture}[zoomboxarray]
    \node [image node] { \includegraphics[width=0.45\textwidth]{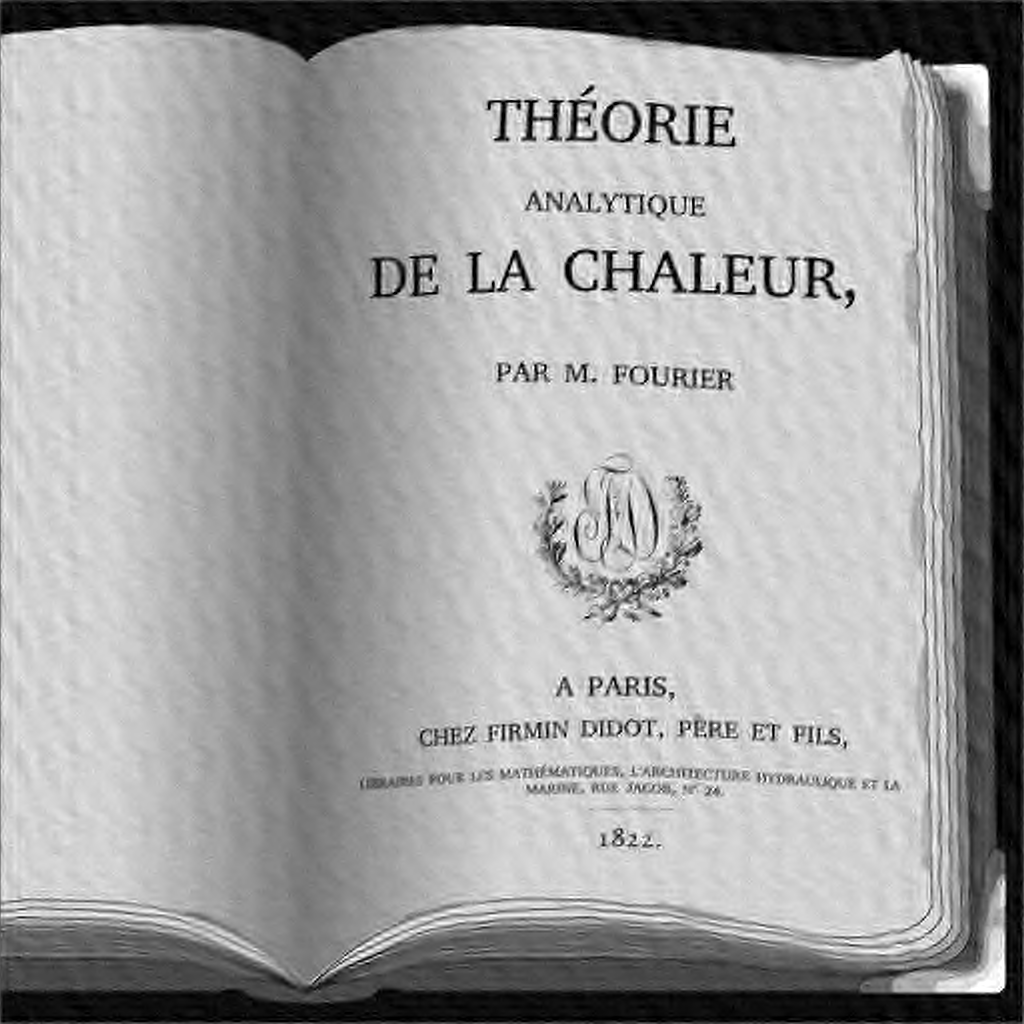} };
    \zoombox[color code=red]{0.905,0.1}
    \zoombox[magnification=2]{0.6,0.47}
    \zoombox[color code=blue,magnification=2]{0.55,0.8}
	\zoombox[color code=green,magnification=2]{0.6,0.15}
\end{tikzpicture}
\end{subfigure}

\begin{subfigure}[b]{0.9\textwidth}
\begin{tikzpicture}[zoomboxarray]
    \node [image node] { \includegraphics[width=0.45\textwidth]{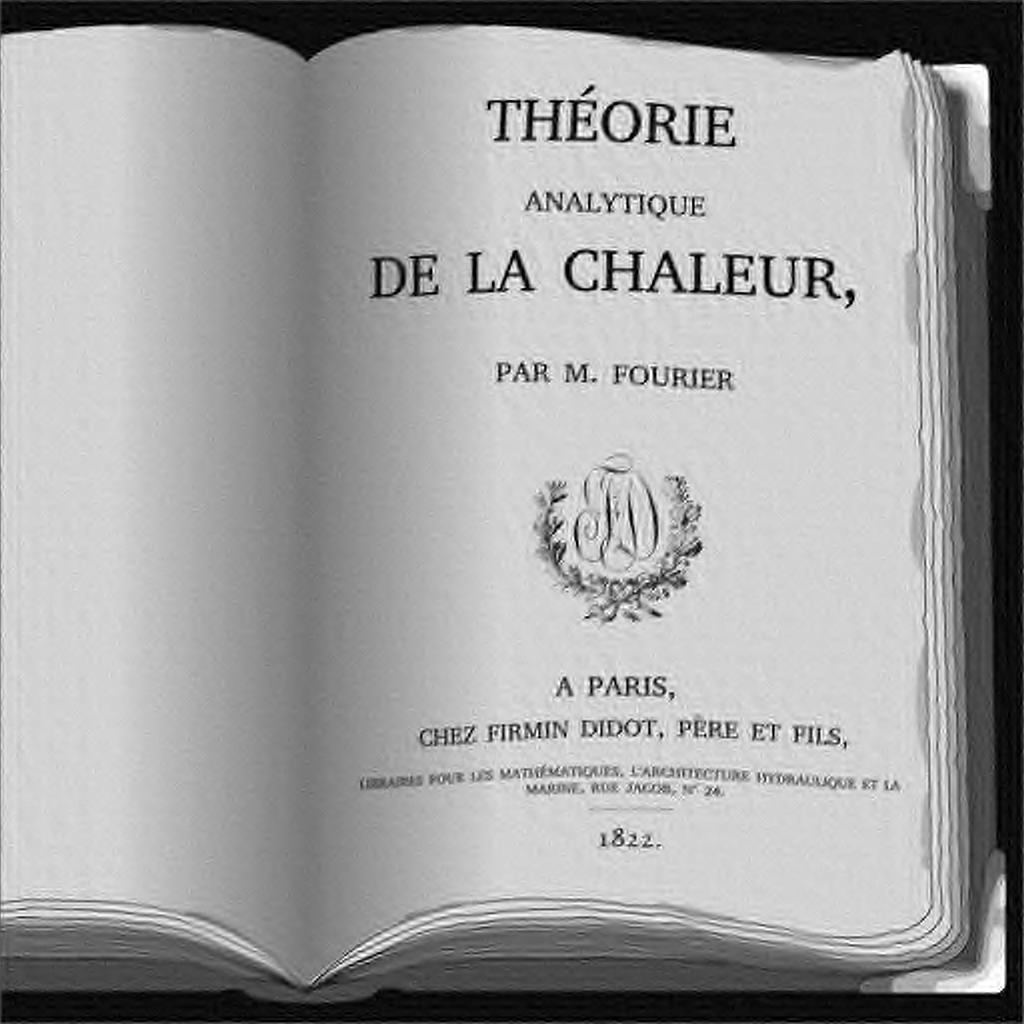} };
    \zoombox[color code=red]{0.905,0.1}
    \zoombox[magnification=2]{0.6,0.47}
    \zoombox[color code=blue,magnification=2]{0.55,0.8}
	\zoombox[color code=green,magnification=2]{0.6,0.15}
\end{tikzpicture}
\end{subfigure}
\caption{A deconvolution example showing the importance of the thresholding strategy. The book image on Figure \ref{fig:original} is blurred with the kernel in Figure \ref{fig:PSF_movement} and degraded with a noise level of $5.10^{-3}$ (see Figure \ref{fig:deconv_book}). Matrices have been constructed with the same number of coefficients that corresponds to $57$ operations per pixel. Top: the result for the simple thresholding strategy, pSNR = 23.71dB. Bottom: the weighted strategy pSNR = 24.07dB.} \label{fig:deconv_compare_thresh}
\end{figure}

\subsection{Approximation in wavelet bases}

In this paragraph, we illustrate the influence of the approximation on the deblurring quality. 
We compare the solution of the original problem \eqref{eq:l1problem} with the solution of the approximated problem \eqref{eq:l1problemThetaApprox} for different numbers of coefficients $K$. 
Computing the exact gradient $\nabla F = \Psi^*H^*H\Psi$ requires two fast Fourier transforms, two fast wavelet transforms and a multiplication by a diagonal matrix. Its complexity is therefore: $2 N \log_2(N) + 2lN + N$, with $l$ denoting the wavelet filter size. The number of operations per pixel is therefore $2 \log_2(N) + 2l + 1$.
The approximate gradient $\nabla F_K$ requires two matrix-vector products with a $K$-sparse matrix. Its complexity is therefore $2\frac{K}{N}$ operations per pixel.
Figure \eqref{fig:pSNR_vs_nnz} displays the restoration quality with respect to the number of operations per pixel.

For the smooth PSF in Figure \ref{fig:PSF_anisotiop_gauss}, the standard approach requires $89$ operations per pixel, while the wavelet method requires $20$ operations per pixel to obtain the same pSNR. This represents an acceleration of a factor $4.5$. For users ready to accept a decrease of pSNR of $0.2$dB, $K$ can be chosen even significantly lower, leading to an acceleration factor of $40$ and around $2.2$ operations per pixels! For the less regular PSF \ref{fig:PSF_movement}, the gain is less important. 
To obtain a similar pSNR, the proposed approach is in fact slower with $138$ operations per pixel instead of $89$ for the standard approach. 
However, accepting a decrease of pSNR of $0.2$dB, our method leads to an acceleration by a factor $1.1$. 
To summarize, the proposed approximation does not really lead to interesting acceleration factors for motion blurs.
Note however that the preconditioners can be used even if the operator is not expanded in the blur domain. 

The different behavior between the two blurs was predicted by Theorem \ref{thm:proof_thresh}, since the compressibility of operators in wavelet bases strongly depends on the regularity $M$ of the PSF. 

\begin{figure}[htp]
	\centering
	\begin{subfigure}[b]{0.45\textwidth}
		\input{images/deconvolution_pSNR_vs_nnz_10_confocal_anisotropic.tex}
	\end{subfigure}\qquad 
	\begin{subfigure}[b]{0.45\textwidth}
		\input{images/deconvolution_pSNR_vs_nnz_10_book_gs_motion.tex}
	\end{subfigure}
	\caption{Evolution of the pSNR of the deconvolved image w.r.t. the number of operations per pixel per iteration. The grey vertical line gives the number of operations per pixel per iteration to solve the exact $\ell^1$-problem \ref{eq:l1problem} using FFTs and FWTs. The horizontal line gives the pSNR obtained using the exact operator.} \label{fig:pSNR_vs_nnz}
\end{figure}
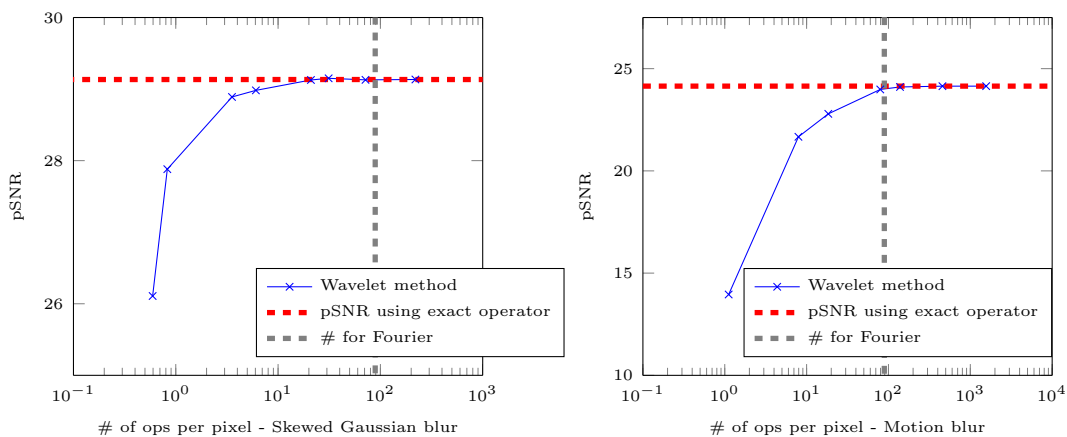

\begin{figure}\centering
\begin{subfigure}[b]{0.9\textwidth}
\begin{tikzpicture}[zoomboxarray]
    \node [image node] { \includegraphics[width=0.45\textwidth]{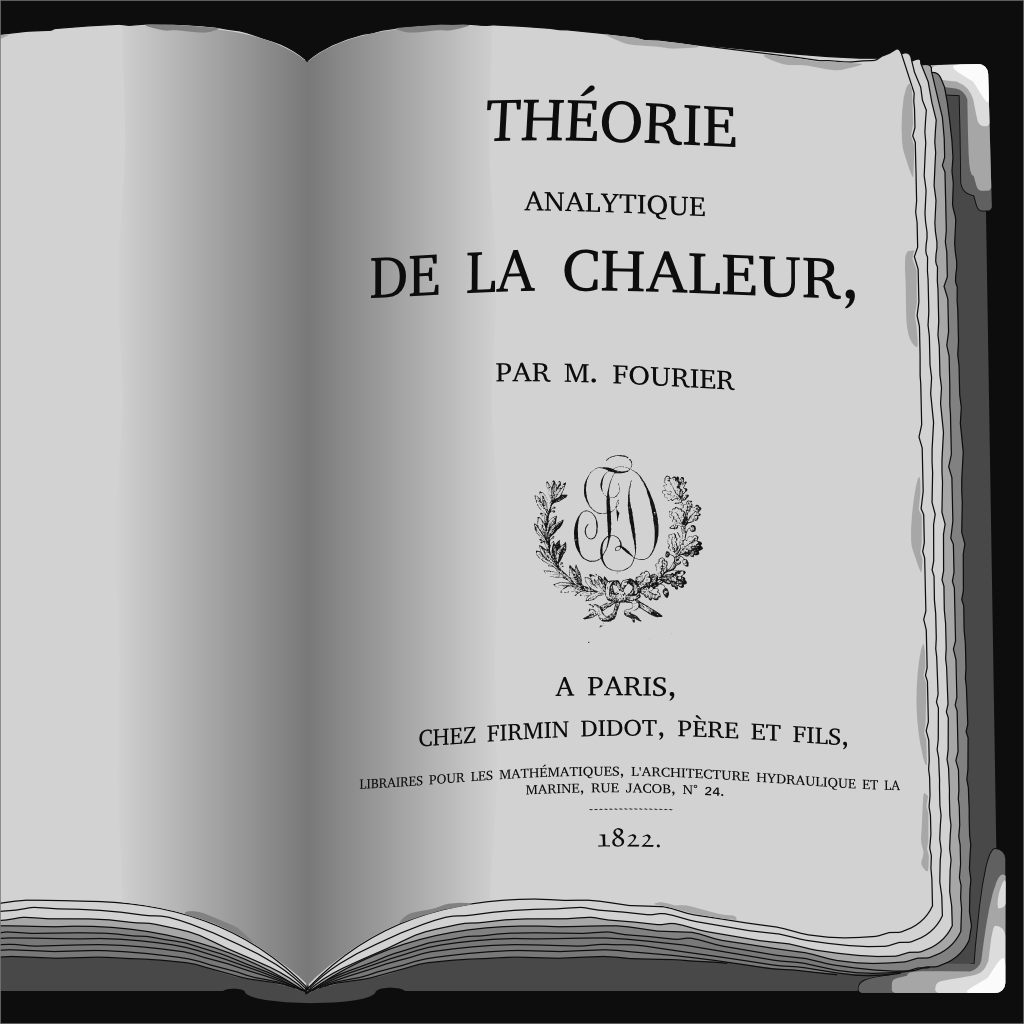} };
    \zoombox[color code=red]{0.905,0.1}
    \zoombox[magnification=2]{0.6,0.47}
    \zoombox[color code=blue,magnification=2]{0.55,0.8}
	\zoombox[color code=green,magnification=2]{0.6,0.15}
\end{tikzpicture}
\end{subfigure}

\begin{subfigure}[b]{0.9\textwidth}
\begin{tikzpicture}[zoomboxarray]
    \node [image node] { \includegraphics[width=0.45\textwidth]{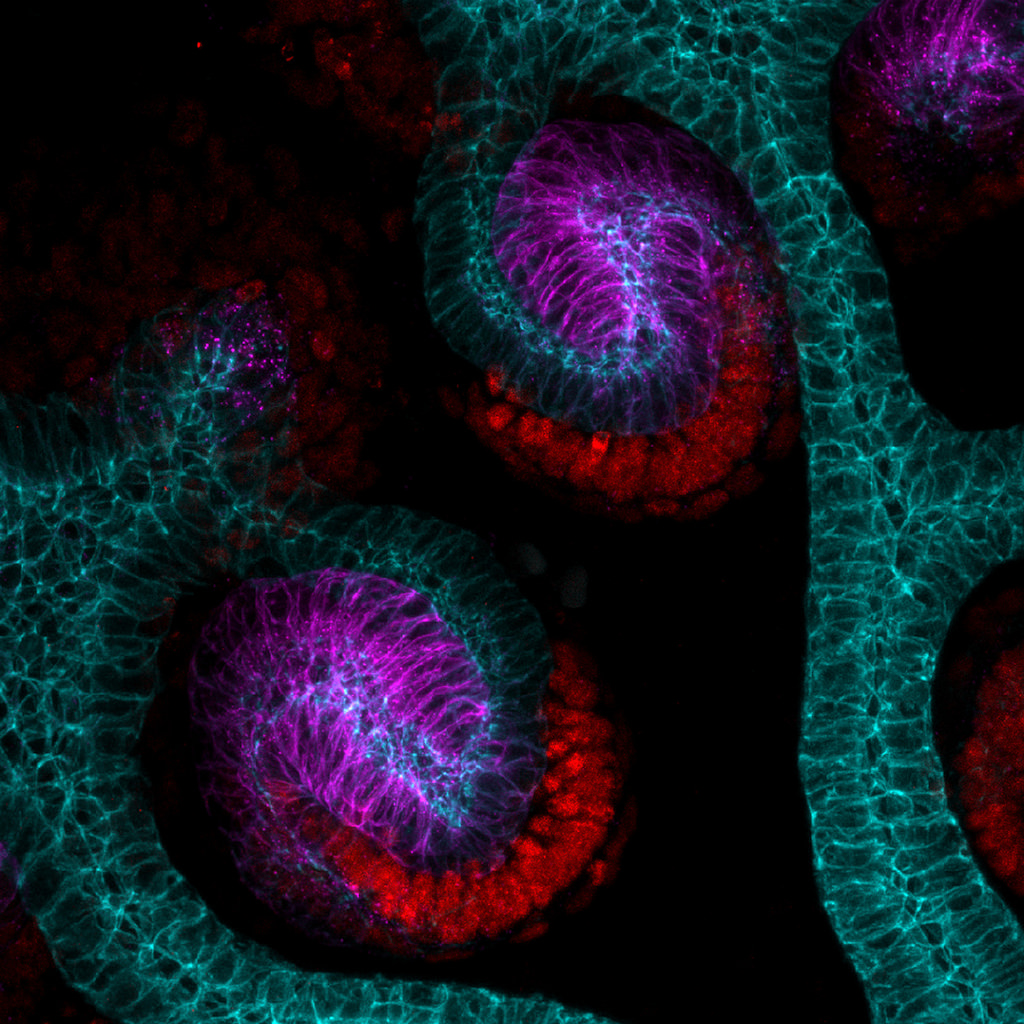} };
    \zoombox[color code=red]{0.9,0.1}
    \zoombox[magnification=2]{0.875,0.5}
    \zoombox[color code=blue,magnification=2]{0.5,0.8}
	\zoombox[color code=green,magnification=2]{0.4,0.2}
\end{tikzpicture}
\end{subfigure}
\caption{Original images $1024 \times 1024$.} \label{fig:original}
\end{figure}

\subsection{Comparing preconditioners}

We now illustrate the interest of using the preconditioners described in Section \ref{sec:optimization}. We compare the cost function w.r.t. the iterations number for different methods: ISTA, FISTA, FISTA with a Jacobi preconditioner (see \eqref{eq:Jacobi_Precond}) and FISTA with a SPAI preconditioner (see \eqref{eq:SPAI_Precond}). For the Jacobi preconditioner, we optimized $\epsilon$ by trial and error in order to maximimize the convergence speed. 

As can be seen in Figure \ref{fig:comp_preconditioners}, the Jacobi and SPAI preconditioners allow reducing the iterations number significantly. We observed that the SPAI preconditioner outperformed the Jacobi preconditioner for all blurs we tested and thus recommend SPAI in general. From a practical point view, a speed-up of a factor 3 is obtained for both blurs.

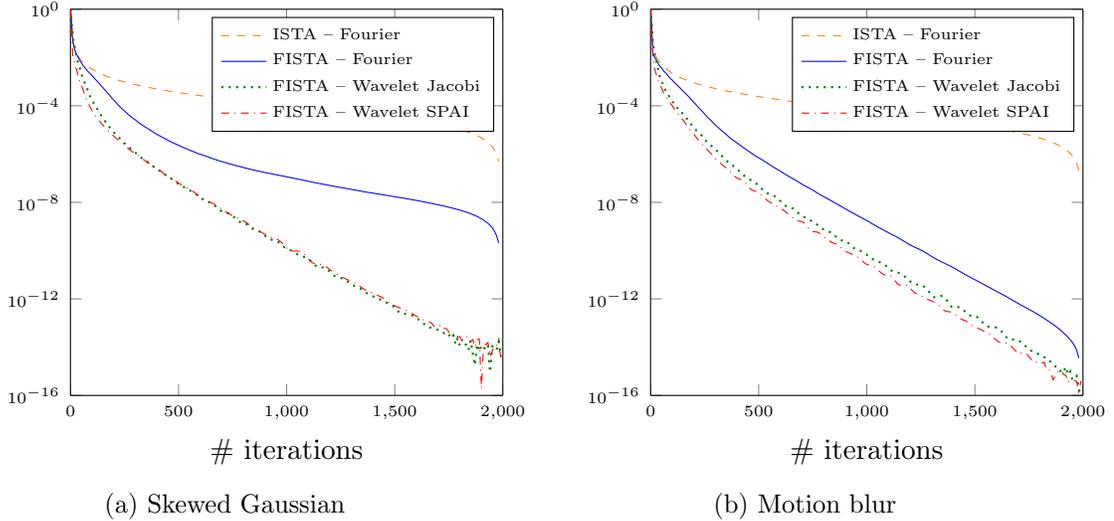
\begin{figure}
\centering
	\begin{subfigure}[b]{0.4\textwidth}
			\input{images/2015-11-30_10-54-18/deconvolution_CF_vs_iteration.tex}
			\caption{Skewed Gaussian}
	\end{subfigure}
	\qquad\qquad
	\begin{subfigure}[b]{0.4\textwidth}
		\input{images/2015-11-24_19-08-45/deconvolution_CF_vs_iteration.tex}
		\caption{Motion blur} 
	\end{subfigure}
	
	\caption{Cost function with respect to iterations for different preconditioners.}\label{fig:comp_preconditioners}
\end{figure}

\subsection{Computing times} \label{sec:computing_times}

In this paragraph, we time precisely what can be gained using the proposed approach on the two examples of Figure \ref{fig:deconv_confocal} and \ref{fig:deconv_book}. The proposed approach consists of:
\begin{itemize}
 \item Finding a number $K$ such that deconvolving the image with matrix $\Theta_K$ instead of $\Theta$ leads to a deacrease of pSNR of less than $0.2$dB.
 \item For each optimization method, finding a number of iterations $Nit$ leading to a precision 
 \begin{equation}
  E(x^{(Nit)}) - E(x^*) \leq 10^{-3} E(x^{(0)}).
 \end{equation}
\end{itemize}

In all experiments, matrix $\Theta_K$ is computed offline, meaning that we assume it is known beforehand. 
The results are displayed in Table \ref{tab:deconv_confocal} for the skewed Gaussian blur and in Table \ref{tab:deconv_book} for the motion blur. For the skewed Gaussian, the total speed-up is roughly 162, which can be decomposed as: sparsification = 7.8, preconditioning = 2.7, GPU = 7.7. For the motion blur, the total speed-up is roughly 32, which can be decomposed as: sparsification = 1.01, preconditioning = 3, GPU = 10.5. 

As can be seen from this example, the proposed sparsification may accelerate computations significantly for smooth enough blurs. On these these two examples, the preconditioning led to an acceleration of a factor 3. Finally, GPU programming allows accelerations of a factor 7-8, which is on par with what is usually reported in the literature.

Note that for the smooth blurs encountered in microscopy, the total computing time is $0.17$ seconds for a $1024\times 1024$ image, which can be considered as real-time.

\begin{table}
\centering
\begin{tabular}{|l|*{4}{c|}}
\hline
                  & Exact  & \backslashbox{FISTA}{GPU} & \backslashbox{Jacobi}{GPU} & \backslashbox{SPAI}{GPU} \\ \hline
Iterations number & 117    &  127                      &  55                        & 43 \\ \hline
Time (in seconds) & 24.30  & \backslashbox{2.57}{0.43} & \backslashbox{1.31}{0.19}  & \backslashbox{1.16}{0.15} \\ \hline
\end{tabular}
\caption{Timing and iterations number depending on the method. The number of operations per pixel is 2.46. This experiment corresponds to the Skewed Gaussian blur in Figure \ref{fig:deconv_confocal}.}\label{tab:deconv_confocal}
\end{table}

\begin{table}
\centering
\begin{tabular}{|l|*{4}{c|}}
\hline
                  & Exact  & \backslashbox{FISTA}{GPU} & \backslashbox{Jacobi}{GPU} & \backslashbox{SPAI}{GPU}  \\ \hline
Iterations number & 107    & 107                       & 52                         & 36                        \\ \hline
Time (in seconds) & 20.03  & \backslashbox{16.7}{1.82} & \backslashbox{7.48}{0.89}  & \backslashbox{6.54}{0.62} \\ \hline
\end{tabular}
\caption{Timing and iterations number depending on the method. The number of operations per pixel is 39.7. This experiment corresponds to the motion blur in Figure \ref{fig:deconv_book}.}\label{tab:deconv_book}
\end{table}

\subsection{Dependency on the blur kernel}

In this paragraph, we analyze the method behavior with respect to different blur kernels. We consider $5$ different types of kernels commonly encountered in applications: Gaussian blur, skewed Gaussian blur, motion blur, Airy pattern and defocus blur. For each type, we consider two different widths ($\sigma=2.5$ and $\sigma=5$). The blurs are shown in Figure \ref{fig:different_blurs}. Table \ref{tab:different_blurs} summarizes the acceleration provided by using simultaneously the sparse wavelet approximation, SPAI preconditioner and GPU programming. We used the same protocol as Section \ref{sec:computing_times}. The acceleration varies from 218 (large Airy pattern) to 19 (large motion blur). As expected, the speed-up strongly depends on the kernel smoothness. Of interest, let us mention that the blurs encountered in applications such as astronomy or microscopy (Airy, Gaussian, defocus) all benefit greatly from the proposed approach. The acceleration factor for the least smooth blur, corresponding to the motion blur, still leads to a significant acceleration, showing that the proposed methodology can be used in nearly all types of deblurring applications.

\begin{figure}
 \centering
	\begin{subfigure}[b]{0.18\textwidth}
			\includegraphics[width=\textwidth]{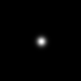}
	\end{subfigure}
	\begin{subfigure}[b]{0.18\textwidth}
			\includegraphics[width=\textwidth]{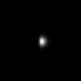}
	\end{subfigure}
	\begin{subfigure}[b]{0.18\textwidth}
			\includegraphics[width=\textwidth]{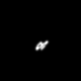}
	\end{subfigure}
	\begin{subfigure}[b]{0.18\textwidth}
			\includegraphics[width=\textwidth]{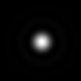}
	\end{subfigure}
	\begin{subfigure}[b]{0.18\textwidth}
			\includegraphics[width=\textwidth]{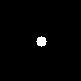}
	\end{subfigure}
	
	\begin{subfigure}[b]{0.18\textwidth}
			\includegraphics[width=\textwidth]{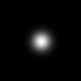}
			\caption{Gaussian}
	\end{subfigure}
	\begin{subfigure}[b]{0.18\textwidth}
			\includegraphics[width=\textwidth]{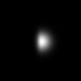}
			\caption{Skewed}
	\end{subfigure}
	\begin{subfigure}[b]{0.18\textwidth}
			\includegraphics[width=\textwidth]{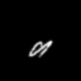}
			\caption{Motion}
	\end{subfigure}
	\begin{subfigure}[b]{0.18\textwidth}
			\includegraphics[width=\textwidth]{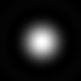}
			\caption{Airy}
	\end{subfigure}
	\begin{subfigure}[b]{0.18\textwidth}
			\includegraphics[width=\textwidth]{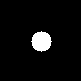}
			\caption{Defocus}
	\end{subfigure}
	\caption{The different blurs used to analyze the method's efficiency.}\label{fig:different_blurs}
\end{figure}

\begin{table}
\centering
\begin{tabular}{|l|c|c|c|c|c|}
  \thickhline
   Blur 			& Time (Fourier) 	& Time (Proposed)	& Speed-up	& \# ops per pixels \\ \thickhline
   Gaussian (small) 		& 14.8			& 0.15			& 99 		& 4 \\ \hline
   Gaussian (large)		& 17.8			& 0.14			& 127 		& 2 \\ \hline	
   Skewed (small) 		& 11.2			& 0.11			& 100 		& 10 \\ \hline
   Skewed (large)		& 10.5			& 0.1			& 102 		& 4  \\ \hline
   Motion (small) 		& 6.0 			& 0.26 			& 23 		& 80 \\ \hline
   Motion (large)		& 9.7			& 0.51			& 19 		& 80 \\ \hline
   Airy (small)			& 15.16			& 0.081			& 187 		& 4 \\ \hline
   Airy (large)			& 18.6 			& 0.085 		& 218		& 2  \\ \hline
   Defocus (small)		& 20.2 			& 0.23			& 87		& 20 \\ \hline
   Defocus (large)		& 21.89 		& 0.20			& 110		& 10 \\ \thickhline
\end{tabular}
\caption{Speed-up of $\ell^1$-$\ell^2$ deconvolution with respect to the different blur kernels, see Figure \ref{fig:different_blurs}. }\label{tab:different_blurs}
\end{table}
\subsection{Dependency on resolution}

In this paragraph, we aim at illustrating that the method efficiency increases with resolution. To this end, we deconvolve the phantom in \cite{guerquin2012realistic} with resolutions ranging from $512\times 512$ to $4096 \times 4096$. The convolution kernel is a Gaussian. Its standard deviation is chosen as $\sigma = 2^L/200$, where $2^L$ is the number of pixels in each direction. This choice is the natural scaling that ensures resolution invariance. We then reproduce the experiment of the previous section to evaluate the speed-up for each resolution. The results are displayed in Table \ref{tab:deconv_resolution}. As can be seen, the speed-up increases significantly with the resolution, which could be expected, since as resolution increases, the kernel's smoothness increases. Of interest, note that 1.35 seconds is enough to restore a $4096\times 4096$ image.

\begin{table}
\centering
\begin{tabular}{|l|c|c|c|c|}
  \thickhline
   Resolution			& 512 	& 1024	& 2048	& 4096	\\ \thickhline
   Time (Fourier) 		& 3.19	& 17.19	& 76	& 352	\\ \hline
   Time Wavelet + GPU + SPAI   	& 0.07	& 0.25 	& 0.55	& 1.35	\\ \hline	
   Total Speed-up  		& 44 	& 70	& 141	& 260	\\ \hline
   Speed-up sparse  		& 4.1	& 4.5	& 9.6	& 9.7	\\ \hline
   Speed-up SPAI 		& 2.4	& 2.2	& 2.1	& 2.7	\\ \hline
   Speed-up GPU			& 4.5	& 7.1	& 7.0	& 10.0	\\ \thickhline
\end{tabular}
\caption{Speed-up of $\ell^1$-$\ell^2$ deconvolution with respect to the image resolution. }\label{tab:deconv_resolution}
\end{table}
%
%
%
%
%

%% file: images/deconvolution_pSNR_vs_nnz_10_confocal_anisotropic.tex
%
%
\begin{tikzpicture}

\begin{axis}[%
width=0.8\textwidth,
height=0.7\textwidth,
at={(0\textwidth,0\textwidth)},
scale only axis,
separate axis lines,
every outer x axis line/.append style={black},
every x tick label/.append style={font=\color{black}\tiny},
xmode=log,
xmin=0.1,
xmax=1000,
xlabel=\# of ops per pixel - Skewed Gaussian blur,
xlabel style={font=\tiny},
xminorticks=true,
every outer y axis line/.append style={black},
every y tick label/.append style={font=\color{black}\tiny},
ymin=25,
ymax=30,
ylabel=pSNR,
ylabel style={font=\tiny},
axis background/.style={fill=white},
legend style={legend cell align=left,align=left, at={(1.2,0.3)}, fill=white, font=\tiny}
]
\addplot [color=blue,solid,mark=x,mark options={solid}]
  table[row sep=crcr]{%
219.90087890625	29.1338958020102\\
71.54345703125	29.1288332943757\\
31.2001953125	29.1503601871284\\
20.94482421875	29.1271191526567\\
6.0673828125	28.9819568393334\\
3.544921875	28.889875732826\\
0.82666015625	27.8808798613276\\
0.59423828125	26.1083460720416\\
};
\addlegendentry{Wavelet method};

\addplot [color=red,dashed,line width=2.0pt]
  table[row sep=crcr]{%
0.01	29.1334684462794\\
1000	29.1334684462794\\
};
\addlegendentry{pSNR using exact operator};

\addplot [color=gray,dashed,line width=2.0pt]
  table[row sep=crcr]{%
89	0\\
89	30\\
};
\addlegendentry{\# for Fourier};

\end{axis}
\end{tikzpicture}%

%% file: images/deconvolution_pSNR_vs_nnz_10_book_gs_motion.tex
%
%
\begin{tikzpicture}

\begin{axis}[%
width=0.8\textwidth,
height=0.7\textwidth,
at={(0\textwidth,0\textwidth)},
scale only axis,
separate axis lines,
every outer x axis line/.append style={black},
every x tick label/.append style={font=\color{black}\tiny},
xmode=log,
xmin=0.1,
xmax=10000,
xlabel=\# of ops per pixel - Motion blur,
xlabel style={font=\tiny},
xminorticks=true,
every outer y axis line/.append style={black},
every y tick label/.append style={font=\color{black}\tiny},
ymin=10,
ymax=27.5,
ylabel=pSNR,
ylabel style={font=\tiny},
axis background/.style={fill=white},
legend style={legend cell align=left,align=left, at={(1.,0.3)}, fill=white, font=\tiny}
]

\addplot [color=blue,solid,mark=x,mark options={solid}]
  table[row sep=crcr]{%
1562.4970703125	24.1448240521283\\
456.66845703125	24.143332771936\\
138.4453125	24.104580891653\\
79.43798828125	23.9824623874007\\
18.41455078125	22.7896973992266\\
7.94921875	21.6663514096869\\
1.1162109375	13.9542006177716\\
};
\addlegendentry{Wavelet method};

\addplot [color=red,dashed,line width=2.0pt]
  table[row sep=crcr]{%
0.1	24.1450599395633\\
10000	24.1450599395633\\
};
\addlegendentry{pSNR using exact operator};

\addplot [color=gray,dashed,line width=2.0pt]
  table[row sep=crcr]{%
89	0\\
89	30\\
};
\addlegendentry{\# for Fourier};

\end{axis}
\end{tikzpicture}%

%% file: images/2015-11-30_10-54-18/deconvolution_CF_vs_iteration.tex
%
%
\definecolor{mycolor1}{rgb}{0.00000,0.44700,0.74100}%
\definecolor{mycolor2}{rgb}{0.85000,0.32500,0.09800}%
\definecolor{mycolor3}{rgb}{0.92900,0.69400,0.12500}%
\definecolor{mycolor4}{rgb}{0.49400,0.18400,0.55600}%
\begin{tikzpicture}

\begin{axis}[%
width=0.951\textwidth,
height=0.85\textwidth,
at={(0\textwidth,0\textwidth)},
scale only axis,
separate axis lines,
every outer x axis line/.append style={black},
every x tick label/.append style={font=\color{black}\tiny},
xmin=0,
xmax=2000,
xlabel={\# iterations},
every outer y axis line/.append style={black},
every y tick label/.append style={font=\color{black}\tiny},
ymode=log,
ymin=1e-16,
ymax=1,
yminorticks=true,
axis background/.style={fill=white},
legend style={legend cell align=left,align=left,draw=black,font=\tiny}
]
\addplot [color=orange,dashed]
  table[row sep=crcr]{%
1	1\\
11	0.0507655200158071\\
21	0.0219209690706774\\
31	0.0136261940685703\\
41	0.00972951911011631\\
51	0.00746635910674212\\
61	0.00600168459837784\\
71	0.004982572077534\\
81	0.00423462074041229\\
91	0.00366373250831419\\
101	0.00321563589556514\\
111	0.00285461037524518\\
121	0.00255823866180953\\
131	0.0023107982729605\\
141	0.00210132605317025\\
151	0.00192174114392175\\
161	0.0017661565927135\\
171	0.0016303020876334\\
181	0.00151101208041842\\
191	0.00140556125622197\\
201	0.00131160750193826\\
211	0.00122761036095582\\
221	0.00115229803063486\\
231	0.00108439675638763\\
241	0.00102292277811993\\
251	0.000967063245639256\\
261	0.000916036557063071\\
271	0.000869269846123785\\
281	0.000826167737727304\\
291	0.000786339520687806\\
301	0.00074947208372307\\
311	0.000715327801356504\\
321	0.000683576464414407\\
331	0.000653956985743948\\
341	0.000626258002215218\\
351	0.00060030666332276\\
361	0.000575987943785773\\
371	0.000553245932367519\\
381	0.000531911389238585\\
391	0.00051189011137154\\
401	0.000493050185001903\\
411	0.000475287462993858\\
421	0.000458510561514011\\
431	0.000442610568673633\\
441	0.000427533262902832\\
451	0.000413204323550887\\
461	0.000399583693661883\\
471	0.000386664326467489\\
481	0.000374376241205276\\
491	0.000362642278931422\\
501	0.000351463320246625\\
511	0.00034079111988574\\
521	0.000330585091468763\\
531	0.000320809845192133\\
541	0.000311443370167004\\
551	0.000302449422201996\\
561	0.000293846369129438\\
571	0.000285585584996635\\
581	0.000277666393992328\\
591	0.00027005293089967\\
601	0.00026272235957187\\
611	0.000255679719235472\\
621	0.000248906841597682\\
631	0.000242389688529595\\
641	0.00023610824172233\\
651	0.000230063081962572\\
661	0.000224227261361625\\
671	0.000218588713483277\\
681	0.000213141047518392\\
691	0.000207877487907142\\
701	0.000202792098871218\\
711	0.000197861890314794\\
721	0.000193084748976105\\
731	0.000188455730672289\\
741	0.00018396346806336\\
751	0.000179608824146595\\
761	0.000175396419797871\\
771	0.0001713170360872\\
781	0.00016734637389805\\
791	0.000163491684673423\\
801	0.000159738325131381\\
811	0.000156087392323587\\
821	0.000152536768837557\\
831	0.000149080235241667\\
841	0.00014572244644833\\
851	0.00014246479547373\\
861	0.000139290932475152\\
871	0.000136202917624074\\
881	0.000133191584309979\\
891	0.000130264401064425\\
901	0.000127415583831508\\
911	0.000124641823060897\\
921	0.00012194077714982\\
931	0.000119311687970589\\
941	0.000116746754976763\\
951	0.000114246669364798\\
961	0.000111805586029873\\
971	0.000109418953805182\\
981	0.000107084750959638\\
991	0.000104800531960529\\
1001	0.000102566314856253\\
1011	0.000100383028376925\\
1021	9.82534816528974e-05\\
1031	9.61721646650123e-05\\
1041	9.41346054500723e-05\\
1051	9.21412332437464e-05\\
1061	9.01954189015901e-05\\
1071	8.82936074988266e-05\\
1081	8.64286132504725e-05\\
1091	8.45992989318847e-05\\
1101	8.28076343508818e-05\\
1111	8.10503706742404e-05\\
1121	7.93277240872603e-05\\
1131	7.76452869600907e-05\\
1141	7.59994182725885e-05\\
1151	7.43892764240938e-05\\
1161	7.28120979479245e-05\\
1171	7.12689390994998e-05\\
1181	6.97535644924712e-05\\
1191	6.82661560062832e-05\\
1201	6.68049389268764e-05\\
1211	6.53739896803131e-05\\
1221	6.39694100519862e-05\\
1231	6.25901402811762e-05\\
1241	6.12323838365297e-05\\
1251	5.98987462427677e-05\\
1261	5.85891512258239e-05\\
1271	5.73022114279368e-05\\
1281	5.60381121977793e-05\\
1291	5.4794224374932e-05\\
1301	5.35727213229516e-05\\
1311	5.23736414060867e-05\\
1321	5.11941427419972e-05\\
1331	5.00324948054916e-05\\
1341	4.88891199047006e-05\\
1351	4.7762804718788e-05\\
1361	4.66554363455681e-05\\
1371	4.55631715476182e-05\\
1381	4.44865280171647e-05\\
1391	4.34281360345828e-05\\
1401	4.2386305762928e-05\\
1411	4.13597410096948e-05\\
1421	4.0348762522136e-05\\
1431	3.93547370346994e-05\\
1441	3.83736045346203e-05\\
1451	3.74042186019232e-05\\
1461	3.64466979451611e-05\\
1471	3.55017631282018e-05\\
1481	3.45703966179639e-05\\
1491	3.36505726224324e-05\\
1501	3.27432432635233e-05\\
1511	3.18472282231141e-05\\
1521	3.09639822801136e-05\\
1531	3.00916208862765e-05\\
1541	2.92305688287183e-05\\
1551	2.83804182119623e-05\\
1561	2.75417336247535e-05\\
1571	2.6712563642773e-05\\
1581	2.58941380632316e-05\\
1591	2.5085933901416e-05\\
1601	2.42880980874965e-05\\
1611	2.35018449272641e-05\\
1621	2.27256298063109e-05\\
1631	2.19580816913827e-05\\
1641	2.12015584231611e-05\\
1651	2.04577711264301e-05\\
1661	1.97256866008108e-05\\
1671	1.90041537738223e-05\\
1681	1.82935533256051e-05\\
1691	1.75926846543092e-05\\
1701	1.69028586749456e-05\\
1711	1.62220819034166e-05\\
1721	1.55493096677587e-05\\
1731	1.48835745409175e-05\\
1741	1.42254826153754e-05\\
1751	1.35753275164946e-05\\
1761	1.293266153422e-05\\
1771	1.2297570000512e-05\\
1781	1.16719128784251e-05\\
1791	1.10543484530511e-05\\
1801	1.04444585520848e-05\\
1811	9.84062569732312e-06\\
1821	9.24277207638432e-06\\
1831	8.65077078198761e-06\\
1841	8.06439652175428e-06\\
1851	7.48434534816202e-06\\
1861	6.91092319331703e-06\\
1871	6.34429386815826e-06\\
1881	5.78356517137323e-06\\
1891	5.22924483323254e-06\\
1901	4.67983720461279e-06\\
1911	4.13581260168623e-06\\
1921	3.59774813049858e-06\\
1931	3.06607725932902e-06\\
1941	2.54045716269854e-06\\
1951	2.02040516075608e-06\\
1961	1.50677225844247e-06\\
1971	9.9881754420814e-07\\
1981	4.97048176732071e-07\\
1991	0\\
};
\addlegendentry{ISTA -- Fourier};

\addplot [color=blue,solid]
  table[row sep=crcr]{%
1	1\\
11	0.0445149779178891\\
21	0.01932993878553\\
31	0.0122748927180821\\
41	0.00825647032470917\\
51	0.005974940488447\\
61	0.00436683004268411\\
71	0.00339728845124634\\
81	0.00265617293587831\\
91	0.00215607634940032\\
101	0.00173340480152602\\
111	0.00137843038682585\\
121	0.0010949035486639\\
131	0.000891280549207833\\
141	0.000701473616814563\\
151	0.000554418000950647\\
161	0.000435982037754637\\
171	0.000345688118682034\\
181	0.000273683049932724\\
191	0.000212789638174379\\
201	0.000169456364099671\\
211	0.000135875885631153\\
221	0.00010837230194769\\
231	8.68073638051192e-05\\
241	7.15622536251568e-05\\
251	5.91858408568584e-05\\
261	4.9418590444639e-05\\
271	4.10992463958477e-05\\
281	3.48069651099324e-05\\
291	2.96828377305156e-05\\
301	2.57304709530838e-05\\
311	2.18222742742857e-05\\
321	1.86541041141817e-05\\
331	1.61187599693913e-05\\
341	1.40203111232921e-05\\
351	1.23457997439907e-05\\
361	1.08979675482496e-05\\
371	9.63403225628789e-06\\
381	8.3915226392165e-06\\
391	7.33443169915365e-06\\
401	6.4839548151144e-06\\
411	5.79284479704385e-06\\
421	5.19305540249908e-06\\
431	4.66052866957629e-06\\
441	4.18099231676721e-06\\
451	3.72206281239773e-06\\
461	3.32114181012852e-06\\
471	3.00425300357767e-06\\
481	2.74083743360492e-06\\
491	2.50663211143571e-06\\
501	2.28163257380218e-06\\
511	2.06886857510086e-06\\
521	1.88435437875039e-06\\
531	1.72778602551051e-06\\
541	1.58536473174641e-06\\
551	1.45878612179198e-06\\
561	1.34165603326158e-06\\
571	1.23167302553512e-06\\
581	1.13001552502006e-06\\
591	1.04366720207847e-06\\
601	9.75193538859124e-07\\
611	9.14777087023189e-07\\
621	8.5390166684939e-07\\
631	7.91615873612245e-07\\
641	7.31765785168402e-07\\
651	6.78770865158179e-07\\
661	6.31980089686902e-07\\
671	5.8919113571935e-07\\
681	5.51436534598915e-07\\
691	5.17457793097615e-07\\
701	4.85862435390359e-07\\
711	4.56612897292258e-07\\
721	4.29881095967285e-07\\
731	4.05186076503539e-07\\
741	3.82208581895677e-07\\
751	3.61140958854179e-07\\
761	3.42090106616631e-07\\
771	3.24691517116749e-07\\
781	3.08110288380466e-07\\
791	2.91800867618425e-07\\
801	2.75984176347305e-07\\
811	2.61392094421565e-07\\
821	2.48140675104187e-07\\
831	2.35990779583223e-07\\
841	2.24877038315549e-07\\
851	2.14618951744877e-07\\
861	2.0492252984761e-07\\
871	1.95791049790187e-07\\
881	1.87088275586955e-07\\
891	1.78715005111967e-07\\
901	1.71045172552888e-07\\
911	1.64189903103785e-07\\
921	1.57708283317438e-07\\
931	1.51286772494993e-07\\
941	1.4504681588408e-07\\
951	1.39171241003079e-07\\
961	1.33646354661193e-07\\
971	1.28364701828536e-07\\
981	1.23227985229298e-07\\
991	1.18277117888278e-07\\
1001	1.13535158094068e-07\\
1011	1.08968862443651e-07\\
1021	1.04529010905912e-07\\
1031	1.0027260979505e-07\\
1041	9.6196869450948e-08\\
1051	9.22344719380184e-08\\
1061	8.83730280059574e-08\\
1071	8.4627071698922e-08\\
1081	8.10089379455009e-08\\
1091	7.75511419599074e-08\\
1101	7.42290546810838e-08\\
1111	7.10063122415162e-08\\
1121	6.79769676827006e-08\\
1131	6.52382145431318e-08\\
1141	6.27941823941946e-08\\
1151	6.03792518672993e-08\\
1161	5.79787309247141e-08\\
1171	5.56999962940933e-08\\
1181	5.35723147803345e-08\\
1191	5.15407224584945e-08\\
1201	4.95861701482702e-08\\
1211	4.77266304169517e-08\\
1221	4.59308211946873e-08\\
1231	4.42438688388679e-08\\
1241	4.26695758939179e-08\\
1251	4.11428946771378e-08\\
1261	3.96485112171177e-08\\
1271	3.82030639944159e-08\\
1281	3.68072247802641e-08\\
1291	3.54558932310939e-08\\
1301	3.41542678302434e-08\\
1311	3.28974653504112e-08\\
1321	3.16720340387665e-08\\
1331	3.04748194340469e-08\\
1341	2.93225212841018e-08\\
1351	2.82082726320274e-08\\
1361	2.71276058396685e-08\\
1371	2.61961961357128e-08\\
1381	2.53701523542481e-08\\
1391	2.46195113537398e-08\\
1401	2.38718504439818e-08\\
1411	2.31203333017157e-08\\
1421	2.23677824627849e-08\\
1431	2.16240078342897e-08\\
1441	2.09064377915524e-08\\
1451	2.02195327347068e-08\\
1461	1.95524345595746e-08\\
1471	1.88965939201638e-08\\
1481	1.82533183150545e-08\\
1491	1.76259758124323e-08\\
1501	1.70146332847899e-08\\
1511	1.645243567856e-08\\
1521	1.59055525067785e-08\\
1531	1.53647833914679e-08\\
1541	1.48344714134545e-08\\
1551	1.43184225533952e-08\\
1561	1.38112375886217e-08\\
1571	1.33097087353448e-08\\
1581	1.28180821531433e-08\\
1591	1.23637042073499e-08\\
1601	1.19213721810745e-08\\
1611	1.14824230645839e-08\\
1621	1.10474858833548e-08\\
1631	1.06218265781085e-08\\
1641	1.02053974315198e-08\\
1651	9.79528988008535e-09\\
1661	9.39102850522414e-09\\
1671	8.99488355982437e-09\\
1681	8.60792449198915e-09\\
1691	8.24505426562535e-09\\
1701	7.90712223937504e-09\\
1711	7.57101621722705e-09\\
1721	7.22999857621292e-09\\
1731	6.89347535646958e-09\\
1741	6.56861819416579e-09\\
1751	6.25345225491037e-09\\
1761	5.94502042927979e-09\\
1771	5.63905742689828e-09\\
1781	5.33777610826332e-09\\
1791	5.04384167364504e-09\\
1801	4.75623436627141e-09\\
1811	4.47226308009253e-09\\
1821	4.19149280272e-09\\
1831	3.91557892079221e-09\\
1841	3.64519488710139e-09\\
1851	3.37892485208286e-09\\
1861	3.11513460701226e-09\\
1871	2.85369374565484e-09\\
1881	2.5955532513955e-09\\
1891	2.34126189784558e-09\\
1901	2.09039975792022e-09\\
1911	1.84233083619642e-09\\
1921	1.59700558756575e-09\\
1931	1.35481447524877e-09\\
1941	1.11592214595116e-09\\
1951	8.80027332269785e-10\\
1961	6.46715346764817e-10\\
1971	4.17979350399765e-10\\
1981	2.06800551512826e-10\\
1991	0\\
};
\addlegendentry{FISTA -- Fourier};

\addplot [color=black!50!green,dotted,thick]
  table[row sep=crcr]{%
1	1\\
11	0.0795558942986128\\
21	0.0227667561592596\\
31	0.00793950625287227\\
41	0.00437734043297102\\
51	0.00204174132453147\\
61	0.00121112945095478\\
71	0.000665010619037956\\
81	0.000430115620230617\\
91	0.000242484713734474\\
101	0.000183233501432082\\
111	0.000105919760487569\\
121	8.51832636498393e-05\\
131	5.14298727520174e-05\\
141	4.14115500632808e-05\\
151	2.95331776808977e-05\\
161	2.26564132295101e-05\\
171	1.76911057281116e-05\\
181	1.29385206378761e-05\\
191	1.05845185905234e-05\\
201	7.85428075962507e-06\\
211	6.48360671302812e-06\\
221	5.17605618076389e-06\\
231	4.21219202121608e-06\\
241	3.42579958428492e-06\\
251	2.76725188923389e-06\\
261	2.33449321032373e-06\\
271	1.91801576132056e-06\\
281	1.61175661151884e-06\\
291	1.34165746930281e-06\\
301	1.12461272647349e-06\\
311	9.96845064675931e-07\\
321	8.30813933473905e-07\\
331	6.56264303634453e-07\\
341	5.5174535626391e-07\\
351	4.96383824145303e-07\\
361	4.46809751226616e-07\\
371	3.85962088640922e-07\\
381	3.22941005639113e-07\\
391	2.662490450077e-07\\
401	2.29819057195288e-07\\
411	2.06184442620462e-07\\
421	1.78241694659717e-07\\
431	1.51703967147463e-07\\
441	1.34680619132692e-07\\
451	1.22683881836599e-07\\
461	1.07528267600908e-07\\
471	8.94923698051206e-08\\
481	7.4905094542548e-08\\
491	6.4681858529063e-08\\
501	5.73762686894398e-08\\
511	5.07124848260137e-08\\
521	4.49958973528049e-08\\
531	4.02253062055403e-08\\
541	3.58458805987136e-08\\
551	3.17460708915056e-08\\
561	2.79600029573446e-08\\
571	2.44133536325393e-08\\
581	2.12651663488547e-08\\
591	1.85647446221709e-08\\
601	1.59055649779097e-08\\
611	1.34073258765015e-08\\
621	1.16332358436638e-08\\
631	1.06921093870993e-08\\
641	1.01160962368587e-08\\
651	9.37504614153426e-09\\
661	8.27867195856231e-09\\
671	6.96803188706187e-09\\
681	5.72710486120307e-09\\
691	4.77121386478841e-09\\
701	4.20324391216372e-09\\
711	3.94239607340538e-09\\
721	3.76726208863655e-09\\
731	3.48806950369084e-09\\
741	3.09193389400153e-09\\
751	2.67359429909453e-09\\
761	2.28523318797168e-09\\
771	1.92435003769906e-09\\
781	1.61286090171051e-09\\
791	1.41025651788716e-09\\
801	1.32673086764869e-09\\
811	1.27794704894371e-09\\
821	1.16623031918512e-09\\
831	9.86779473016118e-10\\
841	8.18636769568223e-10\\
851	7.2070149758878e-10\\
861	6.76191114666184e-10\\
871	6.34363613094906e-10\\
881	5.68487461498834e-10\\
891	4.85069767443461e-10\\
901	4.03365610335848e-10\\
911	3.39774652854418e-10\\
921	3.00915460029381e-10\\
931	2.81632248797127e-10\\
941	2.69608021042553e-10\\
951	2.53299172855998e-10\\
961	2.2738830298712e-10\\
971	1.9422372587744e-10\\
981	1.61371346378356e-10\\
991	1.36336356145503e-10\\
1001	1.20960842653682e-10\\
1011	1.1112143894081e-10\\
1021	1.01402148638869e-10\\
1031	8.98848043413136e-11\\
1041	7.84616299305066e-11\\
1051	6.94475317858982e-11\\
1061	6.32446710040055e-11\\
1071	5.8653104883848e-11\\
1081	5.43558895638379e-11\\
1091	4.96200171493738e-11\\
1101	4.40547543380961e-11\\
1111	3.76404793642919e-11\\
1121	3.09396585306053e-11\\
1131	2.50694195020595e-11\\
1141	2.12792578402579e-11\\
1151	1.98855046132591e-11\\
1161	1.9922827302688e-11\\
1171	1.98109321777683e-11\\
1181	1.84914352983369e-11\\
1191	1.59794602558106e-11\\
1201	1.3068557939369e-11\\
1211	1.0612214371379e-11\\
1221	8.98521257129146e-12\\
1231	8.1225599991431e-12\\
1241	7.6776297751305e-12\\
1251	7.29800817900456e-12\\
1261	6.76809893248123e-12\\
1271	6.0699822817435e-12\\
1281	5.29965169771339e-12\\
1291	4.56681400430175e-12\\
1301	3.97633744874429e-12\\
1311	3.52513409517296e-12\\
1321	3.19766700642486e-12\\
1331	2.92777654866543e-12\\
1341	2.65117530114549e-12\\
1351	2.36346234739168e-12\\
1361	2.03395284436863e-12\\
1371	1.73447169407828e-12\\
1381	1.505891496475e-12\\
1391	1.34052888357213e-12\\
1401	1.2522674095481e-12\\
1411	1.16821233635221e-12\\
1421	1.08632124970953e-12\\
1431	9.62293211418464e-13\\
1441	8.46896804884568e-13\\
1451	7.63036247333912e-13\\
1461	6.98165279649033e-13\\
1471	6.38862322308921e-13\\
1481	5.77176548314716e-13\\
1491	5.14202108170849e-13\\
1501	4.13321431662392e-13\\
1511	3.31722118487559e-13\\
1521	2.77914228125251e-13\\
1531	2.47958818759519e-13\\
1541	2.35971792122096e-13\\
1551	2.39254243635386e-13\\
1561	2.28264109679777e-13\\
1571	2.21699206653196e-13\\
1581	2.05554408098939e-13\\
1591	1.91087306984807e-13\\
1601	1.59040854432832e-13\\
1611	1.37595504546001e-13\\
1621	1.17584707542757e-13\\
1631	1.00248500661453e-13\\
1641	8.21099167435671e-14\\
1651	7.94353266216268e-14\\
1661	6.57462880884233e-14\\
1671	6.41658484709131e-14\\
1681	5.97892464531926e-14\\
1691	5.48534119554301e-14\\
1701	5.24219663900298e-14\\
1711	5.02336653811695e-14\\
1721	4.13832035231126e-14\\
1731	3.73956327958561e-14\\
1741	4.17479203579226e-14\\
1751	3.29217729555196e-14\\
1761	2.26610726695305e-14\\
1771	3.04173840231574e-14\\
1781	1.30082337748914e-14\\
1791	2.33175629721886e-14\\
1801	1.17925109921913e-14\\
1811	1.30811771418534e-14\\
1821	1.33000072427395e-14\\
1831	1.25462591174654e-14\\
1841	7.58611016404884e-15\\
1851	1.89895898657761e-14\\
1861	5.4464380664966e-15\\
1871	1.7749552627422e-15\\
1881	1.32270638757775e-14\\
1891	6.63784639354273e-15\\
1901	1.13548507904193e-14\\
1911	8.24260046670691e-15\\
1921	0\\
1931	1.31541205088155e-14\\
1941	9.23949314852102e-16\\
1951	1.09171905886472e-14\\
1961	1.11360206895332e-14\\
1971	7.92651254320488e-15\\
1981	1.68256033125699e-14\\
1991	3.98757072725644e-15\\
};
\addlegendentry{FISTA -- Wavelet Jacobi};

\addplot [color=red,dashdotted]
  table[row sep=crcr]{%
1	1\\
11	0.011457402906086\\
21	0.00479637439118694\\
31	0.00266183377246633\\
41	0.00118099956793886\\
51	0.000687151004252503\\
61	0.000393406520907938\\
71	0.000218723971791414\\
81	0.000133065173504948\\
91	8.92578774442395e-05\\
101	5.9360222823583e-05\\
111	4.32135179307313e-05\\
121	3.2843404914652e-05\\
131	2.52771184615835e-05\\
141	1.97266940478681e-05\\
151	1.55061138188245e-05\\
161	1.22920509277986e-05\\
171	9.80883671548711e-06\\
181	8.36339123480066e-06\\
191	6.78256902950098e-06\\
201	5.50038183656524e-06\\
211	4.55328138965809e-06\\
221	3.70737381779395e-06\\
231	3.22772853370357e-06\\
241	2.84246914977823e-06\\
251	2.36079180273317e-06\\
261	1.9148765504476e-06\\
271	1.59293928045345e-06\\
281	1.36969058298351e-06\\
291	1.18362502239169e-06\\
301	1.0382501991086e-06\\
311	9.07505004705851e-07\\
321	7.7042396561789e-07\\
331	6.55851777504317e-07\\
341	5.66068175665647e-07\\
351	4.87494733789974e-07\\
361	4.22524951962802e-07\\
371	3.63768051123527e-07\\
381	3.1083673712797e-07\\
391	2.74385142396115e-07\\
401	2.47980722752982e-07\\
411	2.20993363597884e-07\\
421	1.89482084834056e-07\\
431	1.58559285874058e-07\\
441	1.35806946609672e-07\\
451	1.20291821191183e-07\\
461	1.07499708424843e-07\\
471	9.46931664149966e-08\\
481	8.18581018223416e-08\\
491	7.11484718032149e-08\\
501	6.2411333203945e-08\\
511	5.39831371904652e-08\\
521	4.68099475370195e-08\\
531	4.18772917896169e-08\\
541	3.8264810152179e-08\\
551	3.47498498331131e-08\\
561	3.08336094518306e-08\\
571	2.67489960605763e-08\\
581	2.2933470336699e-08\\
591	1.96388048456422e-08\\
601	1.70530532283875e-08\\
611	1.50668925619645e-08\\
621	1.34112201769768e-08\\
631	1.19864513087736e-08\\
641	1.08393463027457e-08\\
651	9.87867506003501e-09\\
661	8.89257522468334e-09\\
671	7.80098126369296e-09\\
681	6.73295523162464e-09\\
691	5.85922175965772e-09\\
701	5.16159220041011e-09\\
711	4.52052161492425e-09\\
721	3.92207113363506e-09\\
731	3.43479743098776e-09\\
741	3.07804219448043e-09\\
751	2.79899894892135e-09\\
761	2.52974042386437e-09\\
771	2.25155265811301e-09\\
781	2.01113093157494e-09\\
791	1.83531779543079e-09\\
801	1.68669536915778e-09\\
811	1.51003888611926e-09\\
821	1.29959420881246e-09\\
831	1.09796301107666e-09\\
841	9.41140631892159e-10\\
851	8.28436954064121e-10\\
861	7.4021985392166e-10\\
871	6.61722044081141e-10\\
881	5.89214610994554e-10\\
891	5.23924413465289e-10\\
901	4.65647720669421e-10\\
911	4.14041601524405e-10\\
921	3.72673789051722e-10\\
931	3.45518602600295e-10\\
941	3.28879904508338e-10\\
951	3.10650749243141e-10\\
961	2.81006929626616e-10\\
971	2.40344746971467e-10\\
981	1.9669964258223e-10\\
991	1.57944953287594e-10\\
1001	1.28103262630956e-10\\
1011	1.08317739059346e-10\\
1021	9.80691960011843e-11\\
1031	9.50703239841866e-11\\
1041	9.57685378927469e-11\\
1051	9.61849229458217e-11\\
1061	9.30222930699943e-11\\
1071	8.46476651090862e-11\\
1081	7.18016330678627e-11\\
1091	5.73051357768455e-11\\
1101	4.45532736223699e-11\\
1111	3.56706680539384e-11\\
1121	3.08226573699981e-11\\
1131	2.88450411193798e-11\\
1141	2.81841255857927e-11\\
1151	2.7569918121517e-11\\
1161	2.6349964623533e-11\\
1171	2.43960063382661e-11\\
1181	2.18390495527799e-11\\
1191	1.89796452534135e-11\\
1201	1.6188977920181e-11\\
1211	1.38184643806497e-11\\
1221	1.20117058099121e-11\\
1231	1.0632030652737e-11\\
1241	9.49562162438024e-12\\
1251	8.4979022510739e-12\\
1261	7.62756631093887e-12\\
1271	6.89314817790973e-12\\
1281	6.26461949925376e-12\\
1291	5.7166446121795e-12\\
1301	5.2324465422857e-12\\
1311	4.77956549127425e-12\\
1321	4.29320343482723e-12\\
1331	3.77659419554664e-12\\
1341	3.23669170774951e-12\\
1351	2.79232071621696e-12\\
1361	2.44180352350886e-12\\
1371	2.17718930262635e-12\\
1381	1.96990856817598e-12\\
1391	1.80354906259129e-12\\
1401	1.63787036176492e-12\\
1411	1.50214707030428e-12\\
1421	1.37661153576266e-12\\
1431	1.28105572504243e-12\\
1441	1.1569061144731e-12\\
1451	1.03185686904456e-12\\
1461	8.91440887642696e-13\\
1471	7.56957633420408e-13\\
1481	6.55031435318829e-13\\
1491	5.69760639340242e-13\\
1501	5.08925871293928e-13\\
1511	4.65378681217609e-13\\
1521	4.3415892015787e-13\\
1531	3.94720873087078e-13\\
1541	3.70868392090501e-13\\
1551	3.46821395448693e-13\\
1561	3.02787916259294e-13\\
1571	2.70814407074281e-13\\
1581	2.38135778675301e-13\\
1591	2.16058252941467e-13\\
1601	1.90819847972612e-13\\
1611	1.70906308791984e-13\\
1621	1.56633723323085e-13\\
1631	1.3747393226773e-13\\
1641	1.28307382486172e-13\\
1651	1.14910117420816e-13\\
1661	1.08904446874277e-13\\
1671	9.78899984630143e-14\\
1681	7.41104608333998e-14\\
1691	7.26515934941597e-14\\
1701	6.78859601859752e-14\\
1711	8.12832252513306e-14\\
1721	5.81601779243741e-14\\
1731	6.02998500219264e-14\\
1741	5.6847197319058e-14\\
1751	4.10671155996103e-14\\
1761	3.95109904377542e-14\\
1771	2.96879503535371e-14\\
1781	2.92502901517651e-14\\
1791	2.44117134766186e-14\\
1801	2.74753348890229e-14\\
1811	1.82358417405019e-14\\
1821	2.56274362593187e-14\\
1831	2.96393214422291e-14\\
1841	1.877075976489e-14\\
1851	1.69714900464938e-14\\
1861	2.04727716606702e-14\\
1871	1.00904990964111e-14\\
1881	2.05214005719782e-14\\
1891	2.16398655320623e-14\\
1901	1.94515645232021e-16\\
1911	1.39808120010515e-14\\
1921	1.23760579278873e-14\\
1931	1.55612516185616e-14\\
1941	4.01188518291042e-15\\
1951	0\\
1961	1.17681965365372e-14\\
1971	1.20356555487313e-14\\
1981	1.16709387139212e-14\\
1991	3.7687406263704e-15\\
};
\addlegendentry{FISTA -- Wavelet SPAI};

\end{axis}
\end{tikzpicture}%

%% file: images/2015-11-24_19-08-45/deconvolution_CF_vs_iteration.tex
%
%
\definecolor{mycolor1}{rgb}{0.00000,0.75000,0.75000}%
\begin{tikzpicture}

\begin{axis}[%
width=0.951\textwidth,
height=0.85\textwidth,
at={(0\textwidth,0\textwidth)},
scale only axis,
separate axis lines,
every outer x axis line/.append style={black},
every x tick label/.append style={font=\color{black}\tiny},
xmin=0,
xmax=2000,
xlabel={\# iterations},
every outer y axis line/.append style={black},
every y tick label/.append style={font=\color{black}\tiny},
ymode=log,
ymin=1e-16,
ymax=1,
yminorticks=true,
axis background/.style={fill=white},
legend style={legend cell align=left,align=left,draw=black,font=\tiny}
]
\addplot [color=orange,dashed]
  table[row sep=crcr]{%
1	1\\
11	0.0233802954883492\\
21	0.0131624875833247\\
31	0.00898532772894091\\
41	0.00666973176463319\\
51	0.00520739959946706\\
61	0.00421022184854188\\
71	0.00349446400435488\\
81	0.00296075729297436\\
91	0.00255100407458711\\
101	0.00222886872775987\\
111	0.00197034834993514\\
121	0.00175929428427603\\
131	0.00158425829832534\\
141	0.00143723021119431\\
151	0.00131227325813664\\
161	0.00120486996144264\\
171	0.00111172097014535\\
181	0.0010303425998876\\
191	0.000958726063495788\\
201	0.000895249458117898\\
211	0.000838558076699622\\
221	0.000787680030540547\\
231	0.000741776635997191\\
241	0.000700228403059348\\
251	0.000662445505517153\\
261	0.000627906906211174\\
271	0.000596244781235847\\
281	0.000567115958580512\\
291	0.000540213663005539\\
301	0.000515300364167443\\
311	0.000492157548886486\\
321	0.000470605523375578\\
331	0.000450486534685679\\
341	0.000431661065815914\\
351	0.000414005517741692\\
361	0.000397410046076435\\
371	0.000381802139124898\\
381	0.000367087888798745\\
391	0.000353201456479921\\
401	0.000340058026374197\\
411	0.000327607013679084\\
421	0.000315811516149175\\
431	0.000304614873160845\\
441	0.000293977070967677\\
451	0.000283847815174272\\
461	0.000274196289761553\\
471	0.000264995832083561\\
481	0.000256219414691149\\
491	0.000247839791935494\\
501	0.000239829315058514\\
511	0.000232167015638632\\
521	0.000224832930041899\\
531	0.000217809305104683\\
541	0.000211077815591792\\
551	0.000204622166749373\\
561	0.000198427707947563\\
571	0.000192478261997502\\
581	0.000186755736197494\\
591	0.000181250222740362\\
601	0.00017595089683638\\
611	0.000170844923164737\\
621	0.000165924285284486\\
631	0.000161180047870152\\
641	0.000156600406041233\\
651	0.000152182898351348\\
661	0.000147920900978845\\
671	0.000143804741095135\\
681	0.000139826693325919\\
691	0.000135981264823661\\
701	0.000132264508397315\\
711	0.000128668506422481\\
721	0.000125190021168949\\
731	0.000121822119449593\\
741	0.000118557458946027\\
751	0.000115396105672832\\
761	0.000112332479068127\\
771	0.000109359822704178\\
781	0.0001064775264193\\
791	0.000103681362708011\\
801	0.000100968963783877\\
811	9.8337109548446e-05\\
821	9.57835077192191e-05\\
831	9.33058478455378e-05\\
841	9.09012799594403e-05\\
851	8.85645585023271e-05\\
861	8.62948134257289e-05\\
871	8.40896223359933e-05\\
881	8.19474101627984e-05\\
891	7.98643558644122e-05\\
901	7.78398517880676e-05\\
911	7.58703682459936e-05\\
921	7.39529621266196e-05\\
931	7.20866268519245e-05\\
941	7.02719000617952e-05\\
951	6.85050106628878e-05\\
961	6.67851886195651e-05\\
971	6.51097785446105e-05\\
981	6.34767771651218e-05\\
991	6.1887212980774e-05\\
1001	6.03392038507915e-05\\
1011	5.88313623300641e-05\\
1021	5.73628656560849e-05\\
1031	5.59319995531208e-05\\
1041	5.45381259838804e-05\\
1051	5.3178694147808e-05\\
1061	5.18537087871554e-05\\
1071	5.05621958993636e-05\\
1081	4.93017604289088e-05\\
1091	4.80722333360931e-05\\
1101	4.68732765392579e-05\\
1111	4.57031980850205e-05\\
1121	4.45607214240369e-05\\
1131	4.3445584647405e-05\\
1141	4.23570224507917e-05\\
1151	4.12934926924451e-05\\
1161	4.02545389976121e-05\\
1171	3.92401426871106e-05\\
1181	3.82491008346732e-05\\
1191	3.72800780748633e-05\\
1201	3.63341294301094e-05\\
1211	3.54102681772726e-05\\
1221	3.45092701726573e-05\\
1231	3.36288448772187e-05\\
1241	3.27683837317251e-05\\
1251	3.1927059020032e-05\\
1261	3.1104225953569e-05\\
1271	3.02987006367933e-05\\
1281	2.95105023225403e-05\\
1291	2.87393677840381e-05\\
1301	2.79852147185983e-05\\
1311	2.72481571998809e-05\\
1321	2.65266978531959e-05\\
1331	2.58200905293525e-05\\
1341	2.51294783165549e-05\\
1351	2.44538989469587e-05\\
1361	2.37933788037164e-05\\
1371	2.31468297076095e-05\\
1381	2.25134868932032e-05\\
1391	2.18933239029011e-05\\
1401	2.12862907020938e-05\\
1411	2.0692142866125e-05\\
1421	2.01101084705309e-05\\
1431	1.95396134854231e-05\\
1441	1.89805457943279e-05\\
1451	1.84325298691853e-05\\
1461	1.78955276052513e-05\\
1471	1.73694188866744e-05\\
1481	1.68537852692979e-05\\
1491	1.63482232505454e-05\\
1501	1.58523984473757e-05\\
1511	1.5365911503964e-05\\
1521	1.488865678829e-05\\
1531	1.44203206072082e-05\\
1541	1.3961160819556e-05\\
1551	1.35107342376456e-05\\
1561	1.30687155674121e-05\\
1571	1.26349804708617e-05\\
1581	1.22095187139262e-05\\
1591	1.17920194832626e-05\\
1601	1.13825149361344e-05\\
1611	1.09807414299904e-05\\
1621	1.05862902622318e-05\\
1631	1.01988411443112e-05\\
1641	9.81846371648998e-06\\
1651	9.44501748835818e-06\\
1661	9.07842856804894e-06\\
1671	8.7181044265696e-06\\
1681	8.3641359946657e-06\\
1691	8.01647337028785e-06\\
1701	7.67488908097973e-06\\
1711	7.33932296948119e-06\\
1721	7.00986201806601e-06\\
1731	6.68610562111026e-06\\
1741	6.36804583574752e-06\\
1751	6.05570145248474e-06\\
1761	5.74891968247449e-06\\
1771	5.44749664236157e-06\\
1781	5.15123322826089e-06\\
1791	4.86032440465513e-06\\
1801	4.57439132254693e-06\\
1811	4.29319628916092e-06\\
1821	4.0168625767186e-06\\
1831	3.74549312741123e-06\\
1841	3.47919799014589e-06\\
1851	3.21750805747477e-06\\
1861	2.96043835971661e-06\\
1871	2.70793408079065e-06\\
1881	2.45995586016978e-06\\
1891	2.21630746987694e-06\\
1901	1.97704351841396e-06\\
1911	1.74189564718475e-06\\
1921	1.51061623611961e-06\\
1931	1.28322005225441e-06\\
1941	1.05963851776783e-06\\
1951	8.39944119097377e-07\\
1961	6.24455365832532e-07\\
1971	4.12842235753352e-07\\
1981	2.04727766981638e-07\\
1991	0\\
};
\addlegendentry{ISTA -- Fourier};

\addplot [color=blue,solid]
  table[row sep=crcr]{%
1	1\\
11	0.0144162729567086\\
21	0.00960320446262731\\
31	0.00663287473587504\\
41	0.00455093398900302\\
51	0.00372896974069037\\
61	0.00270362504773472\\
71	0.00224460662966332\\
81	0.00173313378076074\\
91	0.00137320811817953\\
101	0.00114134913879469\\
111	0.000902550381939103\\
121	0.000749153592013762\\
131	0.000617469438304091\\
141	0.000492126088148233\\
151	0.000399293807386655\\
161	0.000331504652372298\\
171	0.000267113349382366\\
181	0.000215422194205511\\
191	0.000169479140815884\\
201	0.000135967406503089\\
211	0.000108864459485071\\
221	8.59980492522758e-05\\
231	6.97093703310204e-05\\
241	5.6272000352125e-05\\
251	4.45956741351483e-05\\
261	3.51883481122813e-05\\
271	2.90171744859652e-05\\
281	2.38750090211578e-05\\
291	1.91410741311576e-05\\
301	1.56510534154942e-05\\
311	1.3019100693084e-05\\
321	1.07972747038265e-05\\
331	9.06397717624018e-06\\
341	7.54319449020679e-06\\
351	6.36065798835645e-06\\
361	5.25928405059968e-06\\
371	4.37952911134593e-06\\
381	3.75816864903339e-06\\
391	3.28834879035245e-06\\
401	2.84965379220571e-06\\
411	2.42446285317146e-06\\
421	2.07811806495212e-06\\
431	1.80811954712053e-06\\
441	1.57114638505803e-06\\
451	1.3577424008128e-06\\
461	1.17637545974504e-06\\
471	1.022531149779e-06\\
481	8.92123733761322e-07\\
491	7.81632474913703e-07\\
501	6.91725048561969e-07\\
511	6.14732011492934e-07\\
521	5.420214819739e-07\\
531	4.75009618323157e-07\\
541	4.14925423344447e-07\\
551	3.62533563510134e-07\\
561	3.17480269899445e-07\\
571	2.78489460152552e-07\\
581	2.46350267684688e-07\\
591	2.18775656459451e-07\\
601	1.92928354359926e-07\\
611	1.69396684787861e-07\\
621	1.49156962999233e-07\\
631	1.31848121502672e-07\\
641	1.1672672299382e-07\\
651	1.03679552580196e-07\\
661	9.23636634630964e-08\\
671	8.2245968056911e-08\\
681	7.26269860832025e-08\\
691	6.35572852078749e-08\\
701	5.57553211141606e-08\\
711	4.95665176779435e-08\\
721	4.46372837445007e-08\\
731	4.02679201494926e-08\\
741	3.59777175799954e-08\\
751	3.17531523974276e-08\\
761	2.77959249354131e-08\\
771	2.43472046059736e-08\\
781	2.15255013950033e-08\\
791	1.92115640265755e-08\\
801	1.71937407152864e-08\\
811	1.53589004327989e-08\\
821	1.36936860798449e-08\\
831	1.21899448091656e-08\\
841	1.08505814221941e-08\\
851	9.67956453995012e-09\\
861	8.63293823986034e-09\\
871	7.68226157830292e-09\\
881	6.82373195003191e-09\\
891	6.06023854571314e-09\\
901	5.38607130099823e-09\\
911	4.78832458606983e-09\\
921	4.26354093528339e-09\\
931	3.80611995353766e-09\\
941	3.40053008721452e-09\\
951	3.03143415766294e-09\\
961	2.69857551809869e-09\\
971	2.4078675960771e-09\\
981	2.1574859575078e-09\\
991	1.93672130059036e-09\\
1001	1.73485961656369e-09\\
1011	1.54593351907629e-09\\
1021	1.36970405158946e-09\\
1031	1.20971337323551e-09\\
1041	1.07027274681511e-09\\
1051	9.52992963248805e-10\\
1061	8.54006769975192e-10\\
1071	7.66447335264965e-10\\
1081	6.85613039068707e-10\\
1091	6.1117547547046e-10\\
1101	5.45544498227544e-10\\
1111	4.89915264339561e-10\\
1121	4.42360429997183e-10\\
1131	3.99088357368447e-10\\
1141	3.57500476356759e-10\\
1151	3.17456987295395e-10\\
1161	2.80276988992641e-10\\
1171	2.47103120608158e-10\\
1181	2.18151095004427e-10\\
1191	1.9304339877021e-10\\
1201	1.71566175894962e-10\\
1211	1.53728399939817e-10\\
1221	1.39200801393647e-10\\
1231	1.26969426562795e-10\\
1241	1.15849802074019e-10\\
1251	1.05022393236786e-10\\
1261	9.42214916337154e-11\\
1271	8.35678513553464e-11\\
1281	7.34306042915294e-11\\
1291	6.42556053012473e-11\\
1301	5.63838179038712e-11\\
1311	4.98876911287016e-11\\
1321	4.45652170994838e-11\\
1331	4.00822750041794e-11\\
1341	3.61489928857837e-11\\
1351	3.26068900217298e-11\\
1361	2.9407839941281e-11\\
1371	2.65398795272195e-11\\
1381	2.39587085491621e-11\\
1391	2.15858695083017e-11\\
1401	1.93558906371585e-11\\
1411	1.72573666080297e-11\\
1421	1.53233244316383e-11\\
1431	1.35896241022941e-11\\
1441	1.20641113698778e-11\\
1451	1.07275718534843e-11\\
1461	9.55419800228975e-12\\
1471	8.52542618982254e-12\\
1481	7.62813368382195e-12\\
1491	6.8452027069524e-12\\
1501	6.15269723899785e-12\\
1511	5.52703844074715e-12\\
1521	4.95446995062276e-12\\
1531	4.43256244850259e-12\\
1541	3.96326819768717e-12\\
1551	3.54483065049507e-12\\
1561	3.17048783228637e-12\\
1571	2.83307899784721e-12\\
1581	2.52977116359111e-12\\
1591	2.26127379863731e-12\\
1601	2.02727865088573e-12\\
1611	1.82262861377306e-12\\
1621	1.63901066637801e-12\\
1631	1.46933990328912e-12\\
1641	1.31089587740089e-12\\
1651	1.1650461516018e-12\\
1661	1.0338335076662e-12\\
1671	9.17563751248942e-13\\
1681	8.14504799120982e-13\\
1691	7.22785120674669e-13\\
1701	6.41433477150216e-13\\
1711	5.70249260038046e-13\\
1721	5.0879700208566e-13\\
1731	4.55625961266472e-13\\
1741	4.08612623113241e-13\\
1751	3.65941725258807e-13\\
1761	3.26536831798116e-13\\
1771	2.90087244185846e-13\\
1781	2.56830267610156e-13\\
1791	2.26829509647255e-13\\
1801	1.9995041580901e-13\\
1811	1.7587984110486e-13\\
1821	1.54219014960888e-13\\
1831	1.34742864415128e-13\\
1841	1.17326620760403e-13\\
1851	1.01850408180014e-13\\
1861	8.81062788286658e-14\\
1871	7.58691597443934e-14\\
1881	6.49090850747532e-14\\
1891	5.50058747482569e-14\\
1901	4.61472965387106e-14\\
1911	3.83137788842044e-14\\
1921	3.14490535442467e-14\\
1931	2.54430304830934e-14\\
1941	2.0239441460253e-14\\
1951	1.55863026161334e-14\\
1961	1.13343807911703e-14\\
1971	7.3295499353219e-15\\
1981	3.55713137715555e-15\\
1991	0\\
};
\addlegendentry{FISTA -- Fourier};

\addplot [color=black!50!green,dotted,thick]
  table[row sep=crcr]{%
1	1\\
11	0.0404353758038559\\
21	0.0117327646148297\\
31	0.00464590436222921\\
41	0.00182730880363788\\
51	0.00103995146672386\\
61	0.000658856273339839\\
71	0.00039411270434551\\
81	0.000292224594218891\\
91	0.000210414117385713\\
101	0.000146086148176329\\
111	0.000128061130652242\\
121	8.98403080623832e-05\\
131	6.63512123820739e-05\\
141	6.09762344252204e-05\\
151	4.2320459904174e-05\\
161	3.07246905221336e-05\\
171	2.8093045488795e-05\\
181	2.09436555074014e-05\\
191	1.52998306669528e-05\\
201	1.24039168802054e-05\\
211	9.59125657616951e-06\\
221	8.09761867653302e-06\\
231	6.72610462634925e-06\\
241	5.37037389310032e-06\\
251	4.04619385538257e-06\\
261	3.12248606351504e-06\\
271	2.65129153544668e-06\\
281	2.19112957236916e-06\\
291	1.82362991462515e-06\\
301	1.50642779968151e-06\\
311	1.26081920269878e-06\\
321	1.04428963518651e-06\\
331	8.09185265980163e-07\\
341	6.38946549989194e-07\\
351	5.3004050586143e-07\\
361	4.50383593625737e-07\\
371	3.8802604961663e-07\\
381	3.35264725929847e-07\\
391	2.81227555026566e-07\\
401	2.34283245279482e-07\\
411	1.92387857490936e-07\\
421	1.54788170437179e-07\\
431	1.31490585409452e-07\\
441	1.12446455397919e-07\\
451	9.61742122532068e-08\\
461	8.49642010875274e-08\\
471	7.58231349496784e-08\\
481	6.70227131803391e-08\\
491	5.80249031203586e-08\\
501	4.93039922916345e-08\\
511	3.96720186143407e-08\\
521	3.13868826018115e-08\\
531	2.65490959243235e-08\\
541	2.3654374928672e-08\\
551	2.17095828991889e-08\\
561	2.02713847127038e-08\\
571	1.85099609133586e-08\\
581	1.60079761687088e-08\\
591	1.34131047605678e-08\\
601	1.11170358603911e-08\\
611	9.17202695895452e-09\\
621	7.58186877504514e-09\\
631	6.54577851590046e-09\\
641	6.09674571144468e-09\\
651	5.92760038870501e-09\\
661	5.42762152701707e-09\\
671	4.54654561465826e-09\\
681	3.7464727011005e-09\\
691	3.20238197295377e-09\\
701	2.80412631311323e-09\\
711	2.46009675561328e-09\\
721	2.15135419173589e-09\\
731	1.9142422845589e-09\\
741	1.7379473173715e-09\\
751	1.56428979060435e-09\\
761	1.37543561350828e-09\\
771	1.18316126554802e-09\\
781	1.00137994264419e-09\\
791	8.55015454075713e-10\\
801	7.4873193221441e-10\\
811	6.76776268645277e-10\\
821	6.26956100529408e-10\\
831	5.81743389352182e-10\\
841	5.17246932401448e-10\\
851	4.30266184377467e-10\\
861	3.46256707435886e-10\\
871	2.91541908903739e-10\\
881	2.68540555774307e-10\\
891	2.60309092224318e-10\\
901	2.47594519912428e-10\\
911	2.2019042429937e-10\\
921	1.81972403924901e-10\\
931	1.47008716051401e-10\\
941	1.25083760470167e-10\\
951	1.13553120946461e-10\\
961	1.06198059566968e-10\\
971	9.87650991446435e-11\\
981	8.98014810296122e-11\\
991	7.87792584474263e-11\\
1001	6.65611952393339e-11\\
1011	5.55440270630881e-11\\
1021	4.77261803594794e-11\\
1031	4.28707760495422e-11\\
1041	3.93289322946571e-11\\
1051	3.60144364532508e-11\\
1061	3.28717194318335e-11\\
1071	2.99417598571597e-11\\
1081	2.67026174982085e-11\\
1091	2.27842045910388e-11\\
1101	1.87203447278524e-11\\
1111	1.55128113753465e-11\\
1121	1.3699530394176e-11\\
1131	1.29165943472756e-11\\
1141	1.24179043835789e-11\\
1151	1.16730783127009e-11\\
1161	1.05346694451368e-11\\
1171	9.11847140617021e-12\\
1181	7.6446457965247e-12\\
1191	6.32736480707605e-12\\
1201	5.28823768950681e-12\\
1211	4.56231729973329e-12\\
1221	4.13117035391318e-12\\
1231	3.92895005853622e-12\\
1241	3.83620109762311e-12\\
1251	3.7099143787486e-12\\
1261	3.43422529465689e-12\\
1271	2.97662410215405e-12\\
1281	2.39931793294267e-12\\
1291	1.82720323830699e-12\\
1301	1.39395130986778e-12\\
1311	1.17613211422166e-12\\
1321	1.16133604453281e-12\\
1331	1.25504448589555e-12\\
1341	1.3284839027598e-12\\
1351	1.28658495802302e-12\\
1361	1.10752681441871e-12\\
1371	8.52554675545601e-13\\
1381	6.19946753173251e-13\\
1391	4.79209083853249e-13\\
1401	4.36710463835454e-13\\
1411	4.47460561118075e-13\\
1421	4.58949849971587e-13\\
1431	4.39385749885068e-13\\
1441	3.87351069071942e-13\\
1451	3.18975848764558e-13\\
1461	2.56619410155254e-13\\
1471	2.16617397429871e-13\\
1481	2.00624424833185e-13\\
1491	1.98220828400686e-13\\
1501	1.93115021722282e-13\\
1511	1.75658351909064e-13\\
1521	1.47950906271706e-13\\
1531	1.17898607902848e-13\\
1541	9.41098566528941e-14\\
1551	7.80605880094176e-14\\
1561	7.06074809453051e-14\\
1571	6.82846571016786e-14\\
1581	6.77902309516127e-14\\
1591	6.72860141847138e-14\\
1601	6.39669950783306e-14\\
1611	5.71992311925273e-14\\
1621	4.77023328645284e-14\\
1631	3.78162575174182e-14\\
1641	2.89239297788564e-14\\
1651	2.27950036414056e-14\\
1661	2.05578476950677e-14\\
1671	2.06533062091893e-14\\
1681	2.07756889196017e-14\\
1691	2.09347864431377e-14\\
1701	1.87441359267566e-14\\
1711	1.6629362690831e-14\\
1721	1.35991667810211e-14\\
1731	1.1538241937677e-14\\
1741	1.00867829921864e-14\\
1751	8.81400280389793e-15\\
1761	7.08106362445896e-15\\
1771	6.23907057682195e-15\\
1781	5.91842787554157e-15\\
1791	5.30896197768804e-15\\
1801	4.76068743504068e-15\\
1811	4.21730820080981e-15\\
1821	3.86729364903047e-15\\
1831	3.38755342421404e-15\\
1841	3.24069417171921e-15\\
1851	2.33750976887601e-15\\
1861	2.07561076859357e-15\\
1871	1.59831819798538e-15\\
1881	1.53957449698745e-15\\
1891	1.31683796403696e-15\\
1901	8.49336010261756e-16\\
1911	9.96195262756584e-16\\
1921	1.5420221511957e-15\\
1931	7.12267374599916e-16\\
1941	9.7416637488236e-16\\
1951	4.43025411692732e-16\\
1961	0\\
1971	6.41285402560749e-16\\
1981	1.61545177744311e-16\\
1991	1.63992831952558e-16\\
};
\addlegendentry{FISTA -- Wavelet Jacobi};

\addplot [color=red,dashdotted]
  table[row sep=crcr]{%
1	1\\
11	0.00823632486563336\\
21	0.00297528188835979\\
31	0.00141791650495947\\
41	0.000800458777835745\\
51	0.000531003447957077\\
61	0.000336218364953633\\
71	0.000248845892660056\\
81	0.000172165729820507\\
91	0.00012691596807684\\
101	9.74910934979578e-05\\
111	7.14947209894619e-05\\
121	5.94447366116888e-05\\
131	4.08434456703229e-05\\
141	3.05605384550308e-05\\
151	2.52420862160714e-05\\
161	1.90431842262082e-05\\
171	1.33388604561856e-05\\
181	1.03202988276558e-05\\
191	7.75143070015741e-06\\
201	6.54541690852161e-06\\
211	4.95666401096712e-06\\
221	3.82201074769979e-06\\
231	2.97663443973103e-06\\
241	2.33519876887396e-06\\
251	1.85524820823663e-06\\
261	1.56991932653606e-06\\
271	1.25798385850947e-06\\
281	1.01946628453135e-06\\
291	8.17258506541227e-07\\
301	6.31031539441823e-07\\
311	5.30775205186465e-07\\
321	4.44325742341459e-07\\
331	3.5896710767712e-07\\
341	3.07739002166768e-07\\
351	2.64316634569892e-07\\
361	2.22842997280567e-07\\
371	1.83022251858466e-07\\
381	1.43318548947836e-07\\
391	1.13831195705261e-07\\
401	1.00136815792348e-07\\
411	9.25958559373129e-08\\
421	8.44472358703347e-08\\
431	7.42011972092261e-08\\
441	6.11958730055434e-08\\
451	4.70928019569791e-08\\
461	3.66637036439463e-08\\
471	3.16000116465215e-08\\
481	2.94559916474653e-08\\
491	2.76653499673091e-08\\
501	2.4694293257807e-08\\
511	2.08001144109603e-08\\
521	1.71528393755897e-08\\
531	1.40565157745773e-08\\
541	1.17086158800277e-08\\
551	1.015298958597e-08\\
561	9.16693416890119e-09\\
571	8.23409024520759e-09\\
581	7.20338672225308e-09\\
591	6.2240162967477e-09\\
601	5.51518455484361e-09\\
611	4.96285026439214e-09\\
621	4.31411393562717e-09\\
631	3.59788798211453e-09\\
641	2.98114590098602e-09\\
651	2.53142705332282e-09\\
661	2.15705932123311e-09\\
671	1.85183917407915e-09\\
681	1.68725029916442e-09\\
691	1.64955477953378e-09\\
701	1.58178649669746e-09\\
711	1.38538852939651e-09\\
721	1.11200992562826e-09\\
731	8.51666304346026e-10\\
741	6.79320081396749e-10\\
751	6.08905925971428e-10\\
761	6.00735810426514e-10\\
771	5.86176387289476e-10\\
781	5.26796827338364e-10\\
791	4.41796084709602e-10\\
801	3.69865005379337e-10\\
811	3.20644407295358e-10\\
821	2.79965180051291e-10\\
831	2.41496610449104e-10\\
841	2.13633104827083e-10\\
851	1.98337630459504e-10\\
861	1.84144998096188e-10\\
871	1.61378188983607e-10\\
881	1.3269984703739e-10\\
891	1.08898698418871e-10\\
901	9.44492557283723e-11\\
911	8.5761400127059e-11\\
921	7.85688703299564e-11\\
931	7.18944061272631e-11\\
941	6.60228753063515e-11\\
951	5.99372481718446e-11\\
961	5.24695873558098e-11\\
971	4.39038550135628e-11\\
981	3.57472407292617e-11\\
991	2.96839141722461e-11\\
1001	2.6452883339341e-11\\
1011	2.51303961893914e-11\\
1021	2.41505145419328e-11\\
1031	2.23936201707577e-11\\
1041	1.96257638389877e-11\\
1051	1.63044268040694e-11\\
1061	1.31969535555967e-11\\
1071	1.10200637511742e-11\\
1081	1.00121760442649e-11\\
1091	9.78567501914209e-12\\
1101	9.60066907581173e-12\\
1111	8.84095881969336e-12\\
1121	7.37032984644382e-12\\
1131	5.62373275652283e-12\\
1141	4.2797772315127e-12\\
1151	3.72451218844658e-12\\
1161	3.77997848045967e-12\\
1171	3.9456822227038e-12\\
1181	3.83489649793011e-12\\
1191	3.38393089598583e-12\\
1201	2.76414813564453e-12\\
1211	2.19660810672406e-12\\
1221	1.81211345011314e-12\\
1231	1.60580067689999e-12\\
1241	1.48404946127336e-12\\
1251	1.36136324173918e-12\\
1261	1.21702752073306e-12\\
1271	1.07871058142501e-12\\
1281	9.69413030409944e-13\\
1291	8.86608888544943e-13\\
1301	8.15156966897793e-13\\
1311	7.39426545694626e-13\\
1321	6.51922907749791e-13\\
1331	5.53774421653289e-13\\
1341	4.55342007668631e-13\\
1351	3.73308876879228e-13\\
1361	3.21643791851548e-13\\
1371	3.00219474566761e-13\\
1381	2.93696476101782e-13\\
1391	2.83012465482783e-13\\
1401	2.5763029134326e-13\\
1411	2.21206749070335e-13\\
1421	1.84017091030228e-13\\
1431	1.56515248346363e-13\\
1441	1.39876095038699e-13\\
1451	1.29057463438247e-13\\
1461	1.1740662940699e-13\\
1471	1.02610559718136e-13\\
1481	8.65417098409939e-14\\
1491	7.32925576117522e-14\\
1501	6.55408367342335e-14\\
1511	6.11742216267206e-14\\
1521	5.72237077346098e-14\\
1531	5.20077566168351e-14\\
1541	4.62851410779533e-14\\
1551	4.13726990820013e-14\\
1561	3.73242790215606e-14\\
1571	3.30629130450023e-14\\
1581	2.7741712796273e-14\\
1591	2.27631841366984e-14\\
1601	1.89277099923751e-14\\
1611	1.62157091296373e-14\\
1621	1.51876943621735e-14\\
1631	1.53051817641693e-14\\
1641	1.5173008436924e-14\\
1651	1.39956867627571e-14\\
1661	1.22529569664852e-14\\
1671	1.0250775824139e-14\\
1681	8.04788703671658e-15\\
1691	6.71881080163838e-15\\
1701	5.13762618311074e-15\\
1711	5.57575628638697e-15\\
1721	5.64673825842614e-15\\
1731	5.03237705215611e-15\\
1741	5.0005575474489e-15\\
1751	3.87953192007171e-15\\
1761	3.13544504076458e-15\\
1771	2.36688161937498e-15\\
1781	2.39625346987394e-15\\
1791	2.27142310525334e-15\\
1801	2.02910533863687e-15\\
1811	2.18575520796469e-15\\
1821	2.35953865675024e-15\\
1831	1.22137944991532e-15\\
1841	1.25809426303903e-15\\
1851	9.49689832799888e-16\\
1861	4.5281602852572e-16\\
1871	5.94779972604054e-16\\
1881	9.08079711259687e-16\\
1891	1.04025303850503e-15\\
1901	1.08186316004523e-15\\
1911	7.24505645641152e-16\\
1921	6.75552561476209e-16\\
1931	3.45119243362846e-16\\
1941	4.45473065900979e-16\\
1951	0\\
1961	4.13653561193766e-16\\
1971	3.842817106948e-16\\
1981	1.44411598286581e-16\\
1991	3.98967635944283e-16\\
};
\addlegendentry{FISTA -- Wavelet SPAI};

\end{axis}
\end{tikzpicture}%